\documentclass[a4paper,reqno,11pt]{amsart}
\usepackage{amsmath}
\usepackage{amssymb}
\usepackage{cases}
\allowdisplaybreaks[4]

\newtheorem{thm}{Theorem}[section]
\newtheorem{prop}[thm]{Proposition}
\newtheorem{prpty}[thm]{Property}
\newtheorem{lem}[thm]{Lemma}

\newtheorem{cor}[thm]{Corollary}
\newtheorem{prob}[thm]{Problem}

\newtheorem{claim}{Claim}

\theoremstyle{definition}
\newtheorem{definition}[thm]{Definition}
\newtheorem{example}[thm]{Example}
\newtheorem{assum}[thm]{Assumption}

\theoremstyle{remark}
\newtheorem{remark}[thm]{Remark}

\numberwithin{equation}{section}

\newcommand{\bQ}{\mathbb{Q}}
\newcommand{\bP}{\mathbb{P}}
\newcommand\tp{{\tilde{P}}}
\newcommand\mB{{B}}
\newcommand\OO{{\mathcal{O}}}
\newcommand{\rounddown}[1]{\lfloor{#1}\rfloor}
\newcommand{\roundup}[1]{\lceil{#1}\rceil}

\newcommand{\bC}{{\mathbb C}}

\newcommand\BB{\mathcal{B}}

\newcommand\ZZ{{\mathbb{Z}}}

\newcommand\lrw{\longrightarrow}

\newcommand\lcm{{\text{l.c.m.}}}
\newcommand\rr{{\tilde{r}}}

\begin{document}

\title{On the anti-canonical geometry of $\bQ$-Fano threefolds}
\date{\today}
\author{Meng Chen and Chen Jiang}
\address{\rm Department of Mathematics \& LMNS, Fudan University,
Shanghai 200433, China}
\email{mchen@fudan.edu.cn}
\address{\rm Graduate School of Mathematical Sciences, the University of Tokyo,
3-8-1 Komaba, Meguro-ku, Tokyo 153-8914, Japan.}
\email{cjiang@ms.u-tokyo.ac.jp}

\dedicatory{Dedicated to the memory of Professor Gang Xiao}


\begin{abstract}
For a $\bQ$-Fano $3$-fold $X$ on which $K_X$ is a canonical divisor, we investigate the geometry induced from the linear system $|-mK_X|$ and prove that the anti-$m$-canonical map  $\varphi_{-m}$ is birational onto its image for all $m\geq 39$.  By a weak $\bQ$-Fano 3-fold $X$ we mean a projective one with at worst terminal singularities on which $-K_X$ is $\bQ$-Cartier, nef and big. For weak $\bQ$-Fano 3-folds, we prove that  $\varphi_{-m}$ is birational onto its image for all $m\geq 97$. 
\end{abstract} 

\maketitle
\pagestyle{myheadings} \markboth{\hfill  M. Chen \& C. Jiang
\hfill}{\hfill On the anti-canonical geometry of $\bQ$-Fano 3-folds\hfill}


\section{\bf Introduction}
Throughout we work over any algebraically closed field $k$ of characteristic 0 (for instance, $k=\bC$). We adopt the standard notation in Koll\'ar--Mori \cite{KM} and will freely use them.

A normal projective variety $X$ is called a {\it
weak $\bQ$-Fano variety} if $X$ has at worst $\bQ$-factorial terminal singularities and the anti-canonical divisor $-K_X$ is nef and big. A weak
$\bQ$-Fano variety is said to be {\it $\bQ$-Fano} if $-K_X$ is $\bQ$-ample and the Picard number $\rho(X)=1$. According to Minimal Model Program, $\bQ$-Fano varieties form a fundamental class in birational geometry.

Given a $\bQ$-Fano $n$-fold $X$ (resp. weak $\bQ$-Fano $n$-fold $X$), the {\it anti-$m$-canonical map} $\varphi_{-m}$ is the rational map defined by the linear system $|-mK_X|$. By definition, $\varphi_{-m}$ is  birational onto its image when $m$ is sufficiently large.  Therefore it is interesting to find such a practical number $m_n$, independent of $X$, which stably guarantees the birationality of $\varphi_{-m_n}$. Such a number $m_3$ exists due to the boundedness of  $\bQ$-Fano $3$-folds, which was proved by Kawamata \cite{K}, and the boundedness of weak $\bQ$-Fano $3$-folds proved  by Koll\'ar--Miyaoka--Mori--Takagi \cite{KMMT}. It is natural  to consider the following problem.

\begin{prob}\label{b P} {}Find the optimal constant $c$ such that
$\varphi_{-m}$ is birational onto its image for all $m\geq c$ and for all (weak) $\bQ$-Fano
3-folds. 
\end{prob}

The following example tells us that $c\geq 33$.

\begin{example}[{\cite[List 16.6, No.95]{Fletcher}}]
The general weighted hypersurface $X_{66}\subset\mathbb{P}(1,5,6,22,33)$ is a $\mathbb{Q}$-Fano $3$-fold. It is clear that $\varphi_{-m}$ is birational onto its image for $m\geq 33$, but $\varphi_{-32}$ fails to be birational.
\end{example}

It is worthwhile to compare the birational geometry induced from $|mK|$ on varieties of general type with the geometry induced from $|-mK|$ on (weak) $\bQ$-Fano varieties. An obvious feature on Fano varieties is that the behavior of $\varphi_{-m}$ is not necessarily birationally invariant. For example, consider degree $1$ (rational) del Pezzo surface $S_1$ and ${\mathbb P}^2$, $|-K_{{\mathbb P}^2}|$ gives a birational map but $|-K_{S_1}|$ does not. This causes difficulties in studying Problem \ref{b P}.  In fact, even if in dimension 3, there is no known practical upper bound for $c$ in written records. The motivation of this paper is to systematically study $\varphi_{-m}$ on (weak) $\bQ$-Fano 3-folds.

When $X$ is nonsingular, we may take $c=4$ according
to Ando \cite{Ando} and Fukuda \cite{F}. When $X$ has terminal singularities, Problem \ref{b P} was treated by the first author in \cite{C}, where an effective upper bound of $c$ in terms of the Gorenstein index of $X$ is proved (cf. \cite[Theorem 1.1]{C}).  Since, however, the Gorenstein index of a weak $\bQ$-Fano 3-fold can be as large as ``840'' (see Proposition \ref{bound index}),  the number ``$3\times 840+10=2530$'' obtained in \cite[Theorem 1.1]{C} is far from being optimal. It turns out that Problem \ref{b P} is closely related to the following problem (cf. \cite[Theorem 4.5]{C}).

\begin{prob}\label{problem2}
Given a (weak) $\bQ$-Fano 3-fold X, can one find the minimal positive integer $\delta_1=\delta_1(X)$ such that $\dim\overline{\varphi_{-\delta_1}(X)}>1$?
\end{prob}

Problem \ref{problem2} is parallel to the following question on 3-folds of general type:
\begin{quote} 
{\em Let $Y$ be a 3-fold of general type on which $|nK_Y|$ is composed with a pencil of surfaces for some fixed integer $n>0$. Can one find an integer $m$ (bounded from above by a function in terms of $n$) so that $|mK_Y|$ is not composed with a pencil any more?}
\end{quote}
This question was solved by Koll\'ar \cite{Kol86} who proved that one may take $m\leq 11n+5$. The result is a direct application of the semi-positivity of $f_*\omega_{Y/B}^l$ since,  modulo birational equivalence,  one may assume that there is a fibration $f:Y\lrw B$ onto a curve $B$. As far as we know, there is still no known analogy of Koll\'ar's method in treating $\bQ$-Fano varieties.  

Firstly, we shall prove the following theorem. 
\begin{thm}\label{main1}
Let $X$ be a $\mathbb{Q}$-Fano $3$-fold. Then there exists an integer $n_1\leq 10$ such that $\dim\overline{\varphi_{-n_1}(X)}>1$.
\end{thm}

Theorem \ref{main1} is close to be optimal due to the following example.
\begin{example}[{\cite[List 16.7, No.85]{Fletcher}}]
Consider the general codimension $2$ weighted complete intersection $X:=X_{24,30}\subset\mathbb{P}(1,8,9,10,12,15)$ which is a $\mathbb{Q}$-Fano $3$-fold. Then $\dim\overline{\varphi_{-9}(X)}>1$ while $\dim\overline{\varphi_{-8}(X)}=1$.
\end{example}

In fact, theoretically, there are only 4 possible weighted baskets for which we need to take $n_1=10$ (see Remark \ref{optimal remark} and Subsection \ref{exceptional} for more details and discussions).  Theorem \ref{main1} allows us to prove the following result. 

\begin{thm}\label{birationality1} 
Let $X$ be a $\bQ$-Fano 3-fold. Then $\varphi_{-m}$ is birational onto its image for all $m\geq 39$. 
\end{thm}

A key point in proving Theorem \ref{main1} is that we have $\rho(X)=1$, which is not the case for arbitrary weak $\bQ$-Fano 3-folds. Therefore we should study weak $\bQ$-Fano 3-folds in an alternative way.  Our result is as follows. 

\begin{thm}\label{main2}
Let $X$ be a weak $\mathbb{Q}$-Fano $3$-fold. Then $\dim\overline{\varphi_{-n_2}(X)}>1$ for all $n_2\geq 71$.
\end{thm}

Theorem \ref{main2} allows us to study the birationality. 

\begin{thm}\label{birationality2} 
Let $X$ be a weak $\bQ$-Fano 3-fold. Then $\varphi_{-m}$ is birational onto its image for all $m\geq 97$. 
\end{thm}

\begin{remark}
We remark that Theorems \ref{main2} and \ref{birationality2} are true even if $X$ has canonical singularities instead of $\bQ$-factorial terminal singularities, which is not difficult to see. 
\end{remark}

This paper is organized as follows. In Section \ref{preliminaries}, we recall some basic knowledge.  In Section \ref{section non-pencil}, we consider Problem \ref{problem2} on $\bQ$-Fano $3$-folds. We generalize a result of Alexeev and reduce the problem to the numerical behavior of anti-plurigenera, then we apply a method developed by J. A. Chen and the first author to analyze the possible weighted baskets.  Section \ref{section non-pencil weak} is devoted to proving Theorem \ref{main2} for weak $\bQ$-Fano $3$-folds. We reduce the problem to the numerical behavior of Hilbert functions and use Reid's formula to estimate the lower bound of Hilbert functions. {}Finally we study the birationality in  Section \ref{section birationality}.  We give an effective criterion for the birationality of $\varphi_{-m}$. As applications, we prove Theorems \ref{birationality1} and \ref{birationality2} in the last part.
\smallskip

{\noindent
{\bf Acknowledgments.}
The first author appreciates the very effective discussion with Jungkai Chen, Yongnam Lee, Yuri Prokhorov,  De-Qi Zhang and Qi Zhang during the preparation of this paper.  The second author would like to express his gratitude to his supervisor Professor Yujiro Kawamata for  suggestions and encouragement. Part of this paper was written during the second author's visit to Fudan University and he would like to thank for the hospitality and support. The first author was supported by National Natural Science Foundation of China (\#11171068, \#11231003, \#11421061). The second author was supported by Grant-in-Aid for JSPS Fellows (KAKENHI No. 25-6549) and Program for Leading Graduate  Schools, MEXT, Japan.
}

\section{\bf Preliminaries}\label{preliminaries}

Let $X$ be a weak $\bQ$-Fano 3-fold. Denote by $r_X$ the Gorenstein index of $X$, i.e. the Cartier index of $K_X$. 
For any positive integer $m$, the number
$P_{-m}(X):=h^0(X,\OO_X(-mK_X))$ is called the {\it $m$-th anti-plurigenus} of
$X$. Clearly, since $-K_X$ is nef and big, Kawamata--Viehweg
vanishing theorem \cite[Theorem 1-2-5]{KMM} implies
$$h^i(-mK_X)=h^i(X, K_X-(m+1)K_X)=0$$
for all $i>0$ and $m\geq 0$.

For two linear systems $|A|$ and $|B|$, we write $|A|\preceq |B|$ if there exists an effective divisor $F$ such that $$|B|\supset |A|+F.$$
 In particular, if $A\leq B$ as divisors, then $|A|\preceq |B|$.

\subsection{Rational map defined by a Weil divisor}\label{b setting}\

Consider an effective $\bQ$-Cartier Weil divisor $D$ on $X$ with $h^0(X, D)\geq 2$. We study the rational map defined by $|D|$, say 
$$X\overset{\Phi_D}{\dashrightarrow} \bP^{h^0(D)-1}$$ which is
not necessarily well-defined everywhere. By Hironaka's big
theorem, we can take successive blow-ups $\pi: Y\rightarrow X$ such
that:
\begin{itemize}
\item [(i)] $Y$ is nonsingular projective;
\item [(ii)] the movable part $|M|$ of the linear system
$|\rounddown{\pi^*(D)}|$ is base point free and, consequently,
the rational map $\gamma:=\Phi_D\circ \pi$ is a morphism;
\item [(iii)] the support of the
union of $\pi_*^{-1}(D)$ and the exceptional divisors of $\pi$ is of
simple normal crossings.
\end{itemize}
Let $Y\overset{f}\longrightarrow \Gamma\overset{s}\longrightarrow Z$
be the Stein factorization of $\gamma$ with $Z:=\gamma(Y)\subset
\bP^{h^0(D)-1}$. We have the following commutative
diagram.\medskip

\begin{picture}(50,80) \put(100,0){$X$} \put(100,60){$Y$}
\put(170,0){$Z$} \put(170,60){$\Gamma$}
\put(112,65){\vector(1,0){53}} \put(106,55){\vector(0,-1){41}}
\put(175,55){\vector(0,-1){43}} \put(114,58){\vector(1,-1){49}}
\multiput(112,2.6)(5,0){11}{-} \put(162,5){\vector(1,0){4}}
\put(133,70){$f$} \put(180,30){$s$} \put(95,30){$\pi$}
\put(130,10){$\Phi_D$}\put(136,40){$\gamma$}
\end{picture}
\bigskip

{\bf Case $(f_{\rm{np}})$.} If $\dim(\Gamma)\geq 2$, a general
member $S$ of $|M|$ is a nonsingular projective surface by
Bertini's theorem. We say that $|D|$ {\it is not composed with
a pencil of surfaces}.

{\bf Case $(f_{\rm p})$.} If $\dim(\Gamma)=1$, i.e. $\dim\overline{\Phi_D(X)}=1$, then $\Gamma\cong
\bP^1$ since $g(\Gamma)\leq q(Y)=q(X):=h^1(\OO_X)=0$. Furthermore, a
general fiber $S$ of $f$ is an irreducible nonsingular projective surface
by Bertini's theorem. We may write
$$M=\sum_{i=1}^n S_i\sim
nS$$ where $S_i$ is a nonsingular fiber of $f$ for all $i$ and
$n=h^0(D)-1$.  We can write
$$
|D|=|nS'|+E,
$$
where $|S'|=|\pi_*S|$ is an irreducible rational pencil, $|nS'|$ is the movable part and $E$ is the fixed part.
In this case, 
$|D|$ is said to {\it be composed with a rational pencil of surfaces}.  We collect a couple of basic facts about rational pencils as follows.

\begin{lem} Keep the same notation as above. 
If $|D|=|nS'|+E$ is composed with a rational pencil of surfaces, then $n=h^0(D)-1$.
\end{lem}

\begin{lem}\label{pencils}
If $|D_1|=|k_1S_1|+E_1$ and $|D_2|=|k_2S_2|+E_2$ are composed with rational pencils of surfaces and $D_1\leq D_2$, then $|S_1|=|S_2|$.
\end{lem}
\begin{proof}
Since $D_1\leq D_2$, we have ${\rm Mov}|D_1|\preceq {\rm Mov}|D_2|$. Hence $|S_1|\preceq |k_2S_2|$. Thus $|S_1|\preceq |S_2|$ by the irreducibility of $|S_1|$. Then by $h^0(S_1)=h^0(S_2)=2$ and $|S_1|,|S_2|$ are movable, we have $|S_1|= |S_2|$. 
\end{proof}

For another Weil  $\bQ$-Cartier divisor $D'$ satisfying $h^0(X,D')>1$,
we say that $|D|$ and $|D'|$ are {\it composed with the same pencil} if $|D|$ and $|D'|$ are composed with pencils and they define the same fibration structure $Y\rightarrow \bP^1$ on some  model $Y$. In particular, $|D|$ and $|D'|$ are {not composed with the same pencil} if one of them is not composed with  a pencil.

Define
$$\iota=\iota(D):=\begin{cases} 1, & \text{Case\ } (f_{\text{np}});\\
n, & \text{Case\ } (f_{\text{p}}).
\end{cases}$$
Clearly, in both cases, $M\equiv \iota S$ with $\iota\geq 1$.

\begin{definition} For both Case $(f_{\text{np}})$ and Case $(f_{\text{p}})$, we call $S$
{\it a generic irreducible element of $|M|$}.
\end{definition}

We may also define ``a generic irreducible element'' of a moving linear system on a surface in the similar way. 

Restricting our interest to special cases, we fix an
effective Weil divisor $D\sim -m_0K_X$ at the very beginning assuming that $P_{-m_0}\geq 2$ for some integer $m_0>0$. We would like to study the geometry of $X$ induced by $\Phi_D$.

\subsection{Reid's formula}\

A {\it basket} $B$  is a collection of pairs of integers (permitting
weights), say $\{(b_i,r_i)\mid i=1, \cdots, s; b_i\ \text{is coprime
 to}\ r_i\}$.  For simplicity, we will alternatively write a basket as follows, 
 say
$$B=\{(1,2), (1,2), (2,5)\}=\{2\times (1,2), (2,5)\}.$$

Let $X$ be a weak $\bQ$-Fano 3-fold. According to Reid
\cite{YPG},  for a Weil divisor $D$ on $X$, 
$$
\chi(D)=1+\frac{1}{12}D(D-K_X)(2D-K_X)+\frac{1}{12}(D\cdot c_2)+\sum_Qc_Q(D),
$$
where the last sum runs over Reid's basket of orbifold points. If the orbifold point $Q$ is of type $\frac{1}{r}(1,-1,b)$ and $i=i_D$ is the local index of divisor $D$ at $Q$ (i.e. $D\sim iK_X$ around $Q$, $0\leq i< r$),  then
$$
c_Q(D)=-\frac{i(r^2-1)}{12r}+\sum_{j=0}^{i-1}\frac{\overline{jb}(r-\overline{jb})}{2r}.
$$
Here the symbol $\overline{\cdot}$ means the smallest residue mod $r$ and $\sum_{j=0}^{-1}:=0$.   
Write 
\begin{align*}
\chi_{\text{sing}}(D):={}&\sum_Qc_Q(D)\ \text{and}\\
\chi_{\text{reg}}(D):={}&1+\frac{1}{12}D(D-K_X)(2D-K_X)+\frac{1}{12}(D\cdot c_2). 
\end{align*}

We make some remarks here on how to compute the term $c_Q(D)$:
\begin{itemize}
\item[(1)]  If $D=nK_X$ for $n\in {\mathbb Z}$, we take $i=\overline{n}$ (modulo $r$) and then 
$$c_Q(nK_X)=c_Q(iK_X)=-\frac{i(r^2-1)}{12r}+\sum_{j=0}^{i-1}\frac{\overline{jb}(r-\overline{jb})}{2r}.$$

\item[(2)] If $D=tK_X$ for $t\in {\mathbb Z}^+$, then it is easy to see
$$c_Q(tK_X)=-\frac{t(r^2-1)}{12r}+\sum_{j=0}^{t-1}\frac{\overline{jb}(r-\overline{jb})}{2r}.$$

\item[(3)] By Reid's formula,  Kawamata--Viehweg vanishing theorem, and Serre duality,  we have, for any $n>0$, 
\begin{align*}
P_{-n}(X)={}&-\chi(\OO_X((n+1)K_X))\\
={}&\frac{1}{12}n(n+1)(2n+1)(-K_X^3)+(2n+1)-l(-n)
\end{align*}
where
$l(-n)=l(n+1)=\sum_i\sum_{j=1}^n\frac{\overline{jb_i}(r_i-\overline{jb_i})}{2r_i}$ and the sum runs over Reid's basket of orbifold points
$$B_X=\{(b_i,r_i)\mid i=1,\cdots, s; 0<b_i\leq \frac{r_i}{2};b_i \text{ is coprime to } r_i\}.$$
\end{itemize}

The above formula can be rewritten as:
\begin{align*}
{P}_{-1}&{} =
\frac{1}{2}\Big(-K_X^3+\sum_i \frac{b_i^2}{r_i}\Big)-\frac{1}{2}\sum_i b_i+3,\\
{P}_{-m}-{P}_{-(m-1)}&{}= \frac{m^2}{2}\Big(-K_X^3+\sum_i
\frac{b_i^2}{r_i}\Big)-\frac{m}{2}\sum_i b_i+2-\Delta^{m}
\end{align*}
where $\Delta^{m}= \sum_i
\big(\frac{\overline{b_im}(r_i-\overline{b_im})}{2r_i}-
\frac{b_im(r_i-b_im)}{2r_i}\big)$ for any $m\geq 2$.

\subsection{Upper bound of Gorenstein indices}\

The following fact might be known to experts.  We will apply it in our argument. 

\begin{prop}\label{bound index}
Let $X$ be a weak $\bQ$-Fano 3-fold. Then either $r_X= 840$ or $r_X\leq 660$. 
\footnote{This means that the Gorenstein index of a weak $\bQ$-Fano $3$-fold is bounded from above by $840$. Among known $\bQ$-Fano $3$-folds, the maximal Gorenstein index is 420. For example, so is the general weighted hypersurface $X_{19}\subset \mathbb{P}(1,3,4,5,7)$ (cf. \cite[List 16.6, No.40]{Fletcher}). We do not know if this bound is optimal. }\end{prop}
\begin{proof} Write Reid's basket 
$$B_X=\{(b_i,r_i)\mid i=1,\cdots, s; 0<b_i\leq \frac{r_i}{2};b_i \text{ is coprime to } r_i\}.$$ Then, by definition,  $r_X=\text{l.c.m.}\{r_i\mid i=1,\cdots, s\}$.

By \cite{KMMT}, we know that
$(-K_X\cdot c_2(X))\geq 0$. Therefore  Reid \cite[10.3]{YPG} gives the inequality
\begin{align}\label{kwmt2}
\sum_i\big(r_i-\frac{1}{r_i}\big)\leq 24.
\end{align}

Now for the sequence $\mathcal{R}=(r_i)_i$, we define a new set $\mathcal{P}=\{s_j\}_j$ as following: if we factor $r_i$ into its prime factors such that $r_i=p_1^{a_{1i}}p_2^{a_{2i}}\cdots p_k^{a_{ki}}$, then we take $\mathcal{P}=\{p_j^{a_{ji}}\}_{1\leq j\leq k, i}$. 
It is easy to show that if $a,b>1$ and coprime, then 
\begin{align}\label{b ab}
ab-\frac{1}{ab}\geq a-\frac{1}{a}+b-\frac{1}{b}+2. 
\end{align}
So 
\begin{align}\label{b sr}
\sum_j\big(s_j-\frac{1}{s_j}\big)\leq\sum_i\big(r_i-\frac{1}{r_i}\big)\leq 24.
\end{align}
and we also have $\text{l.c.m.}(s_j)_j=\text{l.c.m.}(r_i)_i=r_X$. 
So the problem is reduced to treat the sequence $(s_j)_j$ instead. 
Clearly, for each $j$, 
$$s_j\in \{2,3,4,5,7,8,9,11,13,16,17,19\}.$$
Now we may assume that $r_X>660$. 

Denote by $s_1$ the largest value in $\mathcal{P}$, by $s_2$ the second largest value, by $s_3$, $s_4$ the third, the forth, and so on. For instance, if $\mathcal{P}=\{2,3,4,5\}$, then $s_1=5$, $s_2=4$, $s_3=3$, and $s_4=2$. If the value $s_j$ does not exist by definition, then we set $s_j=1$. In the previous example, we have $s_5=1$. 

Since $\text{l.c.m.}(2,3,4,5,7)=420$ and $\text{l.c.m.}(2,3,4,5,7,8)=840$, if $s_1\leq 8$, then $3,5,7,8\in \mathcal{P}$. In this case $\mathcal{P}=\{3,5,7,8\}$ or $\{2,3,5,7,8\}$ by inequality (\ref{b sr}) and $\mathcal{R}=(3,5,7,8)$ or $(2,3,5,7,8)$ by inequality (\ref{b ab}). In a word, $r_X=840$. 

If $s_1\geq 16$, then
$$
\sum_{j>1}\big(s_j-\frac{1}{s_j}\big)\leq 8+\frac{1}{16}.
$$ 
Then $s_2\leq 8$. Also $s_2\geq 5$ since, otherwise, $\lcm(2,3,4,s_1)\leq 228<r_X$ (a contradiction).  Hence
$$
\sum_{j>2}\big(s_j-\frac{1}{s_j}\big)\leq 3+\frac{1}{16}+\frac{1}{5}.
$$ 
So $s_3\leq 3$, but $2$ and $3$ can not be in $\mathcal{P}$ simultaneously. Then $\lcm (s_j)_j\leq  3\times 8 \times 19 <r_X$, a contradiction. 

If $s_1=13$, then $s_2\geq 5$ since, otherwise, $\lcm(2,3,4,s_1)=12s_1<r_X$ (a contradiction). 
Then
$$
\sum_{j>2}\big(s_j-\frac{1}{s_j}\big)\leq 11-s_2+\frac{1}{13}+\frac{1}{s_2}.
$$ 
If $s_2=11$, then $s_j=1$  for any $j>2$ and $r_X=11\times 13$, a contradiction. 
If $s_2=9$, then $s_3\leq 2$ and $\lcm(s_j)_j\leq 2\times 9\times 13<r_X$, a contradiction.
If $s_2=8$, then $s_3\leq 3$, but $2$ and $3$ can not be in $\mathcal{P}$ simultaneously. So $\lcm(s_j)_j\leq 3\times 8\times 13<r_X$, a contradiction. 
If $s_2=7$, then $s_3\leq 4$, but $3$ and $4$ can not be in $\mathcal{P}$ simultaneously. So $\lcm(s_j)_j\leq 6\times 7\times 13<r_X$, a contradiction. 
If $s_2=5$, then $3$ and $4$ can not be in $\mathcal{P}$ simultaneously. So $\lcm(s_j)_j\leq 6\times 5\times 13<r_X$, a contradiction.

If $s_1=11$, then $9\geq s_2\geq 7$ since, otherwise, $\lcm(2,3,4,5,s_1)=60s_1<r_X$ (a contradiction). Then $$
\sum_{j>2}\big(s_j-\frac{1}{s_j}\big)\leq 6+\frac{1}{11}+\frac{1}{7}.
$$ 
Hence $s_3\leq 5$. 
If $s_3=5$, then $s_j=1$ for any $j>3$ and $\lcm(s_j)_j\leq 5\times 9\times 11<r_X$, a contradiction.
If $s_3=4$, then $s_4\leq 2$ and $\lcm(s_j)_j\leq 4\times 9\times 11<r_X$, a contradiction. 
If $s_3\leq 3$, then $\lcm(s_j)_j\leq 2\times 3\times 9\times 11<r_X$, a contradiction.

If $s_1=9$, then $8\geq s_2\geq 7$ since, otherwise, $\lcm(2,3,4,5,9)=180<r_X$ (a contradiction). Consider firstly the case $s_2=8$.  We have $$
\sum_{j>2}\big(s_j-\frac{1}{s_j}\big)\leq 7+\frac{1}{9}+\frac{1}{8}.
$$ 
If $s_3=7$, then $s_j=1$ for any $j>3$ and $\lcm(s_j)_j\leq 7\times 8\times 9<r_X$, a contradiction. 
If $s_3\leq 5$, then $\lcm(s_j)_j\leq {\rm lcm}(2,3,4,5,8,9)=360<r_X$, a contradiction.
Next we consider the case $s_2=7$. Then $$
\sum_{j>2}\big(s_j-\frac{1}{s_j}\big)\leq 8+\frac{1}{9}+\frac{1}{7}.
$$ 
If $s_3=5$, then $s_4\leq 3$ and $\lcm(s_j)_j\leq 2\times 5\times 7\times 9<r_X$, a contradiction. 
If $s_3\leq 4$, then $\lcm(s_j)_j\leq 4\times 7 \times 9<r_X$, a contradiction. So we conclude the statement. 

From the proof we also know that $r_X=840$ only happens when $\mathcal{R}=(3,5,7,8)$ or $(2,3,5,7,8)$. 
\end{proof}

\section{\bf When is $|-mK_X|$ not composed with a pencil? (Part I)}\label{section non-pencil}

The most important part of this paper is to find a minimal positive integer $m$ so that $|-mK_X|$ is not composed with a pencil of surfaces.  For the convenience of expression, we fix the notation first. 

\begin{definition} Let $X$ be a weak $\bQ$-Fano 3-fold. For any $0\leq i\leq 2$, define
$$\delta_i(X):=\text{min}\{m\in \ZZ^+\mid \dim\overline{\varphi_{-m}(X)}>i\}.$$
\end{definition}

We will mainly treat $\bQ$-Fano 3-folds in this section.

\subsection{Two key theorems}\label{key thms}\

 We prove two theorems here which 
are crucial in proving Theorem \ref{main1}. 

\begin{thm}\label{k1}
Let $X$ be a $\mathbb{Q}$-Fano $3$-fold with the basket $B$ of singularities. Fix a positive integer $m$ such that $P_{-m}>0$.  Assume that, for each pair  $(b,r)\in B$, one of the following conditions is satisfied: 
\begin{itemize}
\item[(1)] $m\equiv 0, \pm 1\mod r$; 

\item[(2)] $m\equiv -2 \mod r$ and $b=\lfloor \frac{r}{2} \rfloor$;

\item[(3)] $m\equiv 2 \mod r$ and $3b\geq r$;

\item[(4)] $m\equiv 3 \mod r$ and $4b\geq r$;

\item[(5)] $m\equiv 4\mod r$, $\overline{b}(r-\overline{b})\geq \overline{4b}(r-\overline{4b})$, and 
$$\overline{b}(r-\overline{b})+\overline{2b}(r-\overline{2b})\geq \overline{3b}(r-\overline{3b})+\overline{4b}(r-\overline{4b}).$$
\end{itemize}
Then one of the following holds:
\begin{itemize}
\item[(I)] $P_{-m}=1$ and $-mK_X\sim E$ is an effective prime divisor;

\item[(II)] $P_{-m}=2$,  $|-mK_X|$ does not have fixed part, and is composed with an irreducible rational pencil of surfaces;

\item[(III)]  $P_{-m}\geq 3$,  $|-mK_X|$ does not have fixed part, and is not composed with a pencil of surfaces.
\end{itemize}
\end{thm}
 
\begin{proof} We generalize the argument of Alexeev \cite[2.18]{A} where the case $m=1$ is treated. 

Assume that none of the conclusions holds, then there exists a strictly effective divisor $E$ such that $-mK_X-E$ is strictly effective and
$$
h^0(-mK_X)-h^0(-mK_X-E)-h^0(E)+h^0(\mathcal{O}_X)=0. 
$$

In fact, if $P_{-m}=1$ and $-mK_X\sim D$ is not a prime divisor, then we take $E$ to be one irreducible component of $D$; if $P_{-m}\geq 2$ and $|-mK_X|$ has fixed part, then we take $E$ to be one component in the fixed part; if $P_{-m}\geq 3$, $|-mK_X|$ does not have fixed part, but  is composed with a (rational) pencil of surfaces, then  $|-mK_X|=|nS|$ with $n\geq 2$ and we can take $E=S$.

By Kawamata--Viehweg vanishing theorem and $\rho(X)=1$, all higher cohomologies vanish for $\OO_X(-mK_X)$, $\OO_X(-mK_X-E)$, $\OO_X(E)$, and $\OO_X$. Hence
$$
\Delta\Delta_\chi(-mK_X, -mK_X-E, E, 0)=0,
$$
where the {\it double difference} of a function $f$ is defined by
$$
\Delta\Delta_f(a, a-d, b, b-d)=f(a)-f(a-d)-f(b)+f(b-d). 
$$
Then we have 
$$
\Delta\Delta_{\chi, {\rm reg}}(-mK_X, -mK_X-E, E, 0)+\Delta\Delta_{\chi, {\rm sing}}(-mK_X, -mK_X-E, E, 0)=0.
$$
It is clear to see that
$$
\Delta\Delta_{\chi, {\rm reg}}(-mK_X, -mK_X-E, E, 0)=\frac{m+1}{2}(-K_X)(-mK_X-E)E>0, 
$$
since $E$ and $-mK_X-E$ are ample by the construction and $\rho(X)=1$.
To get a contradiction, it is sufficient to show that 
$$
\Delta\Delta_{\chi, {\rm sing}}(-mK_X, -mK_X-E, E, 0)\geq 0 
$$
under the assumption of this theorem.  Thus it suffices to show that, for every single point $Q=(b,r)\in B$,
\begin{align}
c_Q(-mK_X)-c_Q(-mK_X-E)-c_Q(E)\geq 0. \label{qq}
\end{align}

Set $F(x):=\frac{\overline{x}(r-\overline{x})}{2r}$ for any integer $x$ and $l:=\overline{m}$. 
We may assume that the local index of $E$ at $Q$ is $i$ ($0\leq i<r$).

Then 
\begin{align}
{}& c_Q(-mK_X)-c_Q(-mK_X-E)-c_Q(E)\notag\\
={}&\Big(-\frac{(2r-l)(r^2-1)}{12r}+\sum_{j=0}^{2r-l-1}F(jb)\Big)\notag\\
{}&-\Big(-\frac{(2r-l-i)(r^2-1)}{12r}+\sum_{j=0}^{2r-l-i-1}F(jb)\Big)\notag\\
{}&-\Big(-\frac{i(r^2-1)}{12r}+\sum_{j=0}^{i-1}F(jb)\Big)\notag\\
={}&\sum_{j=0}^{2r-l-1}F(jb)-\sum_{j=0}^{2r-l-i-1}F(jb)-\sum_{j=0}^{i-1}F(jb)\notag\\
={}&\sum_{j=2r-l-i}^{2r-l-1}F(jb)-\sum_{j=0}^{i-1}F(jb)\notag\\
={}&\sum_{j=l+1}^{l+i}F(jb)-\sum_{j=0}^{i-1}F(jb)\notag\\
={}&\sum_{j=i}^{l+i}F(jb)-\sum_{j=0}^{l}F(jb)\notag\\
={}&\sum_{j=0}^{l}F(ib+jb)-\sum_{j=0}^{l}F(jb). \label{qq2}
\end{align}
Then to prove inequality (\ref{qq}),
it suffices to prove that 
$$
G(x):=\sum_{j=0}^{l}F(x+jb)-\sum_{j=0}^{l}F(jb)\geq 0
$$
for arbitrary integer $x$. 

Note that $G(x)$ is a periodic  piecewise quadratic function with negative leading coefficients. Hence the minimal value can only be reached at end points of each piece. It is easy to see that the set of end points is $\{nr-jb\mid n\in \mathbb{Z}, j=0,1,\ldots,l\}$. Hence $G(x)\geq 0$ is equivalent to $G(-jb)\geq 0$ for all $j=0,1,\ldots,l.$ Note that $G(0)=G(-lb)=0$. 

If $m\equiv 0,1 \mod r$, there is nothing to prove. 

If $m\equiv 2 \mod r$, then $G(-b)=F(b)-F(2b)$. It is easy to see that $F(b)-F(2b)\geq0$ is equivalent to $3b\geq r$. 

If $m\equiv 3 \mod r$, then $G(-b)=G(-2b)=F(b)-F(3b)$. It is easy to see that $F(b)-F(3b)\geq0$ is equivalent to $4b\geq r$.

If $m\equiv 4 \mod r$, then $G(-b)=G(-3b)=F(b)-F(4b)$ and $G(-2b)=F(b)+F(2b)-F(3b)-F(4b)$. 

If $m\equiv -1\mod r$, then $G(x)=\sum_{j=0}^{r-1}F(x+jb)-\sum_{j=0}^{r-1}F(jb)=0$. 

If $m\equiv -2\mod r$, then $G(x)=\sum_{j=0}^{r-2}F(x+jb)-\sum_{j=0}^{r-2}F(jb)=F(b)-F(x+(r-1)b)$. It is easy to see that $F(b)-F(x+(r-1)b)\geq 0$ for all $x$ if and only if $b=\lfloor \frac{r}{2}\rfloor$.  

So we have proved the theorem.
\end{proof}

As a special case of Theorem \ref{k1},  Alexeev proved the following theorem. 

\begin{thm}[{\cite[2.18]{A}}]\label{p3} Let $X$ be a $\bQ$-Fano 3-fold. 
If $P_{-1}\geq 3$, then $|-K_X|$ has no fixed part and is not composed with a pencil of surfaces. 
\end{thm}

Hence we only need to deal with the case when $P_{-1}<3$. For this purpose, we prove the following theorem.

\begin{thm}\label{k2}
Let $X$ be a $\mathbb{Q}$-Fano $3$-fold. Fix a positive integer $m$. Assume that one of the following holds:
\begin{itemize}
\item[(i)] $P_{-m}=1$ and $E\in |-mK_X|$ is an effective prime divisor;

\item[(ii)] $P_{-m}=2$ and $|-mK_X|$ does not have fixed part.
\end{itemize}
Write $n_0:=\min\{n\in \ZZ^+\mid P_{-nm}\geq2\}$. For any integer $l\geq n_0$, write $l=sn_0+t$ with $s\in \mathbb{Z}$ and $0\leq t \leq n_0-1$. Take $$l_0=\min\{l\in \mathbb{Z}_{\geq n_0}\mid P_{-lm}>s+1\}.$$ 

Then $|-l_0mK_X|$ does not have fixed part and is not composed with a pencil of surfaces. 
\end{thm}
\begin{proof}
First we assume that $|-l_0mK_X|$ has a base component $E_{l_0}$. It follows that $P_{-m}=1$ and $E_{l_0}=E$. Thus, by definition, we have $l_0>1$. Hence
\begin{align*}
P_{-(l_0-1)m}={}&h^0(-l_0mK_X-(-mK_X))\\
={}&h^0(-l_0mK_X-E_{l_0})=h^0(-l_0mK_X)>s+1,
\end{align*}
which contradicts the minimality of $l_0$. The similar argument implies that $|-n_0mK_X|$ does not have fixed part. 

Now assume that $|-l_0mK_X|$ is composed with a (rational) pencil of surfaces, i.e. 
$$|-l_0mK_X|=|(P_{-l_0m}-1)S|,$$ 
where $|S|$ is an irreducible rational pencil. Write $l_0=sn_0+t$. Since $P_{-n_0m}\geq 2$, we have  $P_{-sn_0m}\geq s+1$. 

If $t>0$, by the minimality of $l_0$ we get $P_{-sn_0m}= s+1$. So we can write $|{-sn_0mK_X}|= |sS|$ by Lemma \ref{pencils} since $|-n_0mK_X|$ does not have fixed part and $|-sn_0mK_X|\preceq |-l_0mK_X|$.   Now
\begin{align*}
-tmK_X{}&\sim -l_0mK_X-(-sn_0mK_X)\sim (P_{-l_0m}-1)S-sS\\
{}&= (P_{-l_0m}-1-s)S \geq S. 
\end{align*}
This implies that $P_{-tm}\geq 2$, which contradicts the minimality of $n_0$. Hence $t=0$ and $l_0=sn_0$. 

If $s\geq 2$, by the minimality of $l_0$ we get $P_{-(s-1)n_0m}= s\geq 2$. We can write   $|{-(s-1)n_0mK_X}|= |(s-1)S|$ by  Lemma \ref{pencils}. 
Hence
\begin{align*}
-n_0mK_X{}&\sim -l_0mK_X-(-(s-1)n_0mK_X)\sim (P_{-l_0m}-1)S-(s-1)S\\
{}&= (P_{-l_0m}-s)S\geq 2S. 
\end{align*}
This implies that $P_{-n_0m}\geq 3$, which contradicts the minimality of $l_0$. 

Hence $s=1$ and $l_0=n_0$. By $P_{-n_0m}\geq 3$, we have $n_0>1$.  This implies, by assumption, $P_{-m}=1$ and $-mK_X\sim E$ is a fixed prime divisor. Since $E\leq (P_{-Nm}-1)S\sim -n_0mK_X$ and $E$ is reduced and irreducible, $E\leq S_0$ for certain surface $S_0\in |S|$. Hence
\begin{align*}
-(n_0-1)mK_X{}&\sim -n_0mK_X-(-mK_X)\sim (P_{-n_0m}-1)S-E\\
{}&\geq  (P_{-n_0m}-2)S+(S_0-E)\geq S. 
\end{align*}
This implies that $P_{-(n_0-1)m}\geq 2$, which contradicts the minimality of $n_0$.  We are done. \end{proof}

Now let us explain the strategy to prove Theorem \ref{main1}. Firstly, we divide all $\bQ$-Fano 3-folds into several families, roughly speaking, by the value of $P_{-1}$. Then in each family, we may take a suitable $m$ satisfying the condition of Theorem \ref{k1}. Applying Theorem \ref{k2} to $m$, we are able to find the number $l_0$ and so $\delta_1(X)\leq l_0m$. In order to find such $l_0$, or an upper bound of $l_0$, we may assume that $l_0$ is sufficiently large, say, $l_0\geq 9$, then by the assumption of Theorem \ref{k2}, we know the value of $P_{-m}, P_{-2m}, P_{-3m}, \ldots, P_{-8m}$. Then, by Chen--Chen's method (\cite{CC}) on the analysis of baskets,  we can recover all possibilities for baskets of singularities, of which each possibility can be proved to be either impossible or very easy to treat. For this purpose, we need to recall relevant materials on baskets, packings, the canonical sequence, and so on. 

\subsection{Weighted baskets}\label{cc method}\

 All contents of this subsection are mainly from Chen--Chen \cite{CC, explicit}. We list them as follows:
\begin{enumerate}
\item Let $B=\{(b_i,r_i)\mid i=1,\cdots, s; 0<b_i\leq \frac{r_i}{2}; b_i
\text{\ is coprime to\ } r_i\}$ be a basket. We set $\sigma(B):=\sum_i b_i$, $\sigma'(B):=\sum_i\frac{b_i^2}{r_i}$, and $\Delta^n(B)=\sum_i
\big(\frac{\overline{b_in}(r_i-\overline{b_in})}{2r_i}-
\frac{b_in(r_i-b_in)}{2r_i}\big)$ for any integer $n>1$.

\item The new (generalized) basket
$$B':=\{(b_1+b_2, r_1+r_2), (b_3, r_3),\cdots, (b_s,r_s)\}$$ is called a
{\it packing} of $B$, denoted as $B\succeq B'$. Note that $\{(2,4)\}=\{(1,2), (1,2)\}$. We call $B\succ B'$ a {\it prime packing} if $b_1r_2-b_2r_1=1$. A composition of finite
packings is also called a packing. So the relation ``$\succeq$'' is a partial ordering on the set of baskets.

\item Note that for a weak $\bQ$-Fano $3$-fold $X$, all the anti-plurigenera  $P_{-n}$ can be determined by Reid's basket $B_X$ and $P_{-1}(X)$. This leads to the notion of  ``weighted basket''.  We call a pair $\mathbb{B}=(B, \tilde{P}_{-1})$ a {\it weighted basket} if $B$ is
a basket and $\tilde{P}_{-1}$ is a non-negative integer. We write
$(B, \tilde{P}_{-1})\succeq (B',\tilde{P}_{-1})$ if $B\succeq B'$.

\item Given a weighted basket ${\mathbb B}=(B, \tilde{P}_{-1})$, define $\tilde{P}_{-1}({\mathbb B}):=\tilde{P}_{-1}$ and the volume
$$-K^3({\mathbb B}):=2\tilde{P}_{-1}+\sigma(B)-\sigma'(B)-6.$$
For all $m\geq 1$, we define the ``anti-plurigenus'' in the following inductive way:
\begin{align*}
{}&\tilde{P}_{-(m+1)}-\tilde{P}_{-m}\\
={}& \frac{1}{2}(m+1)^2(-K^3({\mathbb B})+\sigma'(B))+2-\frac{m+1}{2}\sigma-\Delta^{m+1}(B).
\end{align*}
Note that, if we set ${\mathbb B}=(B_X, P_{-1}(X))$ for a given weak $\bQ$-Fano 3-fold $X$, then we can verify directly that $-K^3({\mathbb 
B})=-K_X^3$ and $\tilde{P}_{-m}({\mathbb B})=P_{-m}(X)$ for all $m\geq
1$.
\end{enumerate}

\begin{prpty}[{\cite[Section 3]{explicit}}]
Assume ${\mathbb B}:=(B, \tilde{P}_{-1})\succeq \mathbb{B}':=(B',
\tilde{P}_{-1})$. Then
\begin{itemize}
\item[(i)] $\sigma(B)=\sigma(B')$ and
$\sigma'(B)\geq \sigma'(B')$;

\item[(ii)] For all integer $n\geq 1$, $\Delta^n(B)\geq
\Delta^n(B')$;

\item[(iii)] $-K^3({\mathbb B})+\sigma'(B)=-K^3({\mathbb B}')+\sigma'(B')$;

\item[(iv)] $-K^3({\mathbb B})\leq -K^3({\mathbb B}')$;

\item[(v)] $\tilde{P}_{-m}({\mathbb B})\leq \tilde{P}_{-m}({\mathbb B}')$ for all
$m\geq 2$.
\end{itemize}
\end{prpty}

Next we recall the ``canonical'' sequence of a basket $B$. Set
$S^{(0)}:=\{\frac{1}{n}\mid n\geq 2\}$,
$S^{(5)}:=S^{(0)}\cup\{\frac{2}{5}\}$, and inductively for all $n \ge
5$,
$$S^{(n)}:=S^{(n-1)}\cup\big\{\frac{b}{n}\mid 0<b<\frac{n}{2},\ b\
\text{is coprime to}\ n\big\}.$$
Each set $S^{(n)}$ gives a division of
the interval $(0,\frac{1}{2}]=\underset{i}\bigcup
[\omega_{i+1}^{(n)}, \omega^{(n)}_i]$ with
$\omega_{i}^{(n)},\omega_{i+1}^{(n)} \in S^{(n)}$. Let
$\omega_{i+1}^{(n)}=\frac{q_{i+1}}{p_{i+1}}$ and
$\omega^{(n)}_i=\frac{q_i}{p_i}$ with $\text{g.c.d}(q_l,p_l)=1$ for
$l=i,i+1$. Then it is easy to see that $q_ip_{i+1}-p_iq_{i+1}=1$
for all $n$ and $i$ (cf. \cite[Claim A]{explicit}).

Now given a basket ${B}=\{(b_i, r_i)\mid i=1,\cdots,s\}$, we define new baskets $\BB^{(n)}(B)$, where $\BB^{(n)}(\cdot)$ can be regarded as an operator on the set of baskets.  
For each $(b_i,r_i)
\in B$, if $\frac{b_i}{r_i} \in S^{(n)}$, then we set
$\BB^{(n)}_i:=\{(b_i,r_i)\}$. If $\frac{b_i}{r_i}\not\in S^{(n)}$,
then $\omega^{(n)}_{l+1} < \frac{b_i}{r_i} < \omega^{(n)}_{l}$ for
some $l$. We write $\omega^{(n)}_{l}=\frac{q_l}{p_l}$ and
$\omega^{(n)}_{l+1}=\frac{q_{l+1}}{p_{l+1}}$ respectively.
 In this situation, we can unpack $(b_i,r_i)$ to
$\BB^{(n)}_i:=\{(r_i q_l-b_ip_l) \times (q_{l+1},p_{l+1}),(-r_i
q_{l+1}+b_i p_{l+1}) \times (q_l,p_l)\}$. Adding up those
$\BB^{(n)}_i$, we get a new basket $\BB^{(n)}(B)$, which is
uniquely defined according to the construction and $\BB^{(n)}(B)
\succeq B$ for all $n$. Note that, by the definition,  $B=\BB^{(n)}(B)$ for sufficiently large $n$.

Moreover, we have
$$\BB^{(n-1)}(B)=\BB^{(n-1)}(\BB^{(n)}(B))
\succeq \BB^{(n)}(B)$$
 for all $n\geq 1$ (cf. \cite[Claim B]{explicit}). Therefore we have a chain of baskets
 $$\BB^{(0)}(B)
\succeq \BB^{(5)}(B) \succeq \cdots \succeq \BB^{(n)}(B) \succeq \cdots \succeq B. $$
The step $\BB^{(n-1)}(B) \succeq \BB^{(n)}(B)$ can be achieved by a
number of  successive prime packings. Let $\epsilon_n(B)$ be the
number of such prime packings. For any $n>0$, set $B^{(n)}:=\BB^{(n)}(B)$. 

The following properties are essential to represent $B^{(n)}$.

\begin{lem}[{\cite[Lemma 2.16]{explicit}}]\label{delta}For the above sequence $\{B^{(n)}\}$, the following statements hold:
\begin{itemize}
\item[(i)]
$\Delta^j(B^{(0)})= \Delta^j(B)$ for $j=3,4$;
\item[(ii)]
$\Delta^j(B^{(n-1)})= \Delta^j(B^{(n)})$ for all $j <n$;
\item[(iii)]
$\Delta^n(B^{(n-1)})= \Delta^n(B^{(n)})+\epsilon_n(B)$.
\end{itemize}
\end{lem}

It follows that $\Delta^j(B^{(n)})=\Delta^j(B) $ for all $j \leq n$ and
$$\epsilon_n(B)=\Delta^n(B^{(n-1)})-\Delta^n(B^{(n)})=\Delta^n(B^{(n-1)})-\Delta^n(B).$$

Moreover, given a weighted basket ${\mathbb B}=(B, \tilde{P}_{-1})$, we
can similarly consider $\BB^{(n)}({\mathbb B}):=(B^{(n)},
\tilde{P}_{-1})$. It follows that
$$\tilde{P}_{-j}(\BB^{(n)}({\mathbb B}))=\tilde{P}_{-j}({\mathbb
B})$$
for all  $j \leq n$. Therefore we can realize the
canonical sequence of weighted baskets as an approximation of weighted
baskets via anti-plurigenera.

We now recall the relation between weighted baskets and
anti-plurigenera more closely. For a given weighted basket ${\mathbb B}=(B,\tilde{P}_{-1})$, we start by computing the non-negative number $\epsilon_n$ and $B^{(0)}$, $B^{(5)}$  in terms of $\tilde{P}_{-m}$. {}From the definition of $\tilde{P}_{-m}$ we get
\begin{align*}
\sigma(B)={}&10-5\tilde{P}_{-1}+\tilde{P}_{-2},\\
\Delta^{m+1}={}&(2-5(m+1)+2(m+1)^2)+\frac{1}{2}(m+1)(2-
3m)\tilde{P}_{-1}\\
{}&+\frac{1}{2}m(m+1)\tilde{P}_{-2}+
\tilde{P}_{-m}-\tilde{P}_{-(m+1)}.
\end{align*}
In particular, we have
\begin{align*}
\Delta^3={}&5-6\tp_{-1}+4\tp_{-2}-\tp_{-3};\\
\Delta^4={}&14-14\tp_{-1}+6\tp_{-2}+\tp_{-3}-\tp_{-4}.
\end{align*}
Assume $B^{(0)}=\{n_{1,r}^0\times (1,r)\mid r\geq 2\}$. By Lemma
\ref{delta}, we have
\begin{align*}
\sigma(B)={}&\sigma(B^{(0)})=\sum n_{1,r}^0;\\
\Delta^3(B)={}&\Delta^3(B^{(0)})=n_{1,2}^0;\\
\Delta^4(B)={}&\Delta^4(B^{(0)})=2n_{1,2}^0+n_{1,3}^0.
\end{align*}
Thus we get $B^{(0)}$ as follows: $$\begin{cases}
n_{1,2}^0=5-6\tp_{-1}+4\tp_{-2}-\tp_{-3};\\
n_{1,3}^0=4-2\tp_{-1}-2\tp_{-2}+3\tp_{-3}-\tp_{-4};\\
n_{1,4}^0=1+3\tp_{-1}-\tp_{-2}-2\tp_{-3}+\tp_{-4}-\sigma_5;\\
n_{1,r}^0=n_{1,r}^0, r\geq 5,
\end{cases}$$
where $\sigma_5:=\sum_{r\geq 5} n_{1,r}^0$. A computation gives
$$\epsilon_5=2+\tp_{-2}-2\tp_{-4}+\tp_{-5}-\sigma_5.$$
Therefore we get $\mB^{(5)}=\{n_{1,r}^5\times (1,r), n_{2,5}^5\times (2,5)\mid r\geq 2\}$ as follows:
$$\begin{cases}
n_{1,2}^5=3-6\tp_{-1}+3\tp_{-2}-\tp_{-3}+2\tp_{-4}-\tp_{-5}+\sigma_5;\\
n_{2,5}^5=2+\tp_{-2}-2\tp_{-4}+\tp_{-5}-\sigma_5;\\
n_{1,3}^5=2-2\tp_{-1}-3\tp_{-2}+3\tp_{-3}+\tp_{-4}-\tp_{-5}+\sigma_5;\\
n_{1,4}^5=1+3\tp_{-1}-\tp_{-2}-2\tp_{-3}+\tp_{-4}-\sigma_5;\\
n_{1,r}^5=n_{1,r}^0, r\geq 5.
\end{cases}$$
Because $\mB^{(5)}=\mB^{(6)}$, we see $\epsilon_6=0$ and on the
other hand
$$\epsilon_6=3\tp_{-1}+\tp_{-2}-\tp_{-3}-\tp_{-4}-\tp_{-5}+\tp_{-6}-\epsilon=0$$
where $\epsilon:=2\sigma_5-n_{1,5}^0 \ge 0$.

Going on a similar calculation, we get
\begin{align*}
\epsilon_7={}&1+\tilde{P}_{-1}+\tilde{P}_{-2}-\tilde{P}_{-5}-\tilde{P}_{-6}
+\tilde{P}_{-7}-2\sigma_5+2n^0_{1,5}+n^0_{1,6};\\
 \epsilon_{8} ={}&
2\tilde{P}_{-1}+\tilde{P}_{-2}+\tilde{P}_{-3}-\tilde{P}_{-4}-\tilde{P}_{-5}
-\tilde{P}_{-7}+\tilde{P}_{-8}\\
{}&-3\sigma_5+3 n^0_{1,5}+2 n^0_{1,6}+n^0_{1,7}.
\end{align*}


A weighted basket ${\mathbb B}=(B,\tilde{P}_{-1})$ is said to be
{\it geometric} if ${\mathbb B}=(B_X, P_{-1}(X))$ for a 
$\bQ$-Fano 3-fold $X$. Geometric baskets are subject to some geometric properties. 
By \cite{K}, we have that
$(-K_X\cdot c_2(X))> 0$. Therefore \cite[10.3]{YPG} gives the inequality
\begin{align}\label{kwmt}
\gamma(B):=\sum_{i} \frac{1}{r_i}-\sum_{i} r_i+24> 0.
\end{align}
For packings, it is easy to see the following lemma.
\begin{lem}\label{31} Given a packing of baskets $B_1\succeq
 B_2$, we have $\gamma(B_1) \geq \gamma(B_2)$. In particular, if
inequality (\ref{kwmt}) does not hold for $B_1$, then it does not hold for $B_2$.
\end{lem}

Lemma \ref{31} implies that, for two weighted baskets ${\mathbb B}_1\succeq {\mathbb B}_2$, if  ${\mathbb B}_1$ is non-geometric, then neither is ${\mathbb B}_2$. 

Furthermore, $-K^3({\mathbb B})=-K_X^3>0$ gives the
inequality
\begin{align}\label{vol}
\sigma'(B)<2P_{-1}+\sigma(B)-6. 
\end{align}
{}Finally, by \cite[Lemma 15.6.2]{Kollar}, if $P_{-m}>0$ and
$P_{-n} >0$, then
\begin{align}\label{2m}
 P_{-m-n} \ge P_{-m}+P_{-n}-1. 
\end{align}

\subsection*{Notation}
For the connivence of readers who are not familiar with Chen--Chen's method, we collect in the following the notation that will be frequently used in the rest of this paper.

$B=\{(b_i,r_i)\}$: a basket.

$\sigma(B)=\sum_i b_i$.

 $\sigma'(B)=\sum_i\frac{b_i^2}{r_i}$.

$\mathbb{B}=(B, \tp_{-1})$: a weighted basket.

$-K^3(\mathbb{B})$: the volume of $\mathbb{B}$.

$\tp_{-m}(\mathbb{B})$: the $m$-th anti-plurigenus of $\mathbb{B}$. We just write $P_{-m}$ instead if $\mathbb{B}$ is geometric.

$\{B^{(m)}\}$: the canonical sequence of $B$.

$B^{(m)}=\{n^m_{b,r}\times (b,r)\}$: expression of $B^{(m)}$.

$\epsilon_m(B)$: the number of prime packings between $B^{(m-1)}$ to $B^{(m)}$.

$\sigma_5=\sum_{r\geq 5} n_{1,r}^0$.

$\epsilon=2\sigma_5-n_{1,5}^0$.

$\gamma(B)=\sum_{i} \frac{1}{r_i}-\sum_{i} r_i+24$.

Note that usually we will omit $B$ in the symbols if $B$ is clear enough.

\subsection{$\mathbb{Q}$-Fano $3$-folds with  $h^0(-K)=2$}\label{P=2}\

In this subsection we prove the following theorem. 
\begin{thm}\label{p2}
Let $X$ be a $\mathbb{Q}$-Fano $3$-fold with $P_{-1}=2$. Then for any integer $m\geq 6$, $\dim\overline{\varphi_{-m}(X)}>1$. In particular, $\delta_1(X)\leq 6$. 
\end{thm}

Theorem \ref{p2} is optimal due to the following example. 
\begin{example}[{\cite[List 16.6, No.88]{Fletcher}}]
Consider the general weighted hypersurface $X_{42}\subset\mathbb{P}(1^2,6,14,21)$,  which is a $\mathbb{Q}$-Fano $3$-fold with $P_{-1}=2$. Then $\dim\overline{\varphi_{-6}(X_{42})}>1$ while $\dim\overline{\varphi_{-5}(X_{42})}=1$. So $\delta_1(X_{42})=6$. 
\end{example}

\begin{proof}[Proof of Theorem \ref{p2}]
Since $P_{-1}>0$, it is sufficient to prove that there exists an integer $m\leq 6$ such that $\dim\overline{\varphi_{-m}(X)}>1$.

Assume, to the contrary, that $\delta_1(X)>6$. Then, by applying Theorems \ref{k1} and  \ref{k2} to the case $m=1$, we have 
$$
P_{-1}=2, P_{-2}=3,P_{-3}=4,P_{-4}=5,P_{-5}=6,P_{-6}=7. 
$$
Now  by those formulae in Subsection \ref{cc method}, we have $n_{1,2}^0=1$, $n_{1,3}^0=1$, $n_{1,4}^0=\epsilon_5=1-\sigma_5$, and $0=\epsilon_6=1-\epsilon$. Hence $\epsilon=1$, and this implies $\sigma_5=n_{1,5}^0=1$. Hence the basket $B^{(5)}=B^{(0)}=\{(1,2),(1,3),(1,5)\}$ by $\epsilon_5=0$. Since $B^{(5)}$ admits no prime packings, $B=B^{(5)}$ and $-K_X^3=-K^3({\mathbb B}(X))=-1/30<0$, a contradiction. 
\end{proof}

\subsection{$\mathbb{Q}$-Fano $3$-folds with  $h^0(-K)=1$}\label{P=1}\ 

We are going to prove the following theorem. 

\begin{thm}\label{p1}
Let $X$ be a $\mathbb{Q}$-Fano $3$-fold with $P_{-1}=1$. Then, for any integer $m\geq 9$, $\dim\overline{\varphi_{-m}(X)}>1$. In particular, $\delta_1(X)\leq 9$. 
\end{thm}

This result is optimal as well due to the following example. 
\begin{example}[{\cite[List 16.7, No.85]{Fletcher}}]
Consider the general codimension $2$ weighted complete intersection $X=X_{24,30}\subset\mathbb{P}(1,8,9,10,12,15)$ which is a $\mathbb{Q}$-Fano $3$-fold with $P_{-1}=1$. Then $\dim\overline{\varphi_{-9}(X)}>1$ and $\dim\overline{\varphi_{-8}(X)}=1$ since $P_{-8}=2$. So $\delta_1(X)=9$. 
\end{example}

\begin{proof}[Proof of Theorem \ref{p1}]
Since $P_{-1}>0$, it is sufficient to prove that there exists an integer $m\leq 9$ such that $\dim\overline{\varphi_{-m}(X)}>1$. Assume, to the contrary,  that $\delta_1(X)>l$ for some integer $l\leq 9$.  We will deduce a contradiction. 

Applying Theorems \ref{k1} and \ref{k2} to the case $m=1$, we discuss on the number $n_0$ (defined in Theorem \ref{k2}). By Chen--Chen \cite[Theorem 1.1]{CC}, we have $n_0\leq 8$. 

If $n_0=2$ and set $l=6$, then Theorem \ref{k2}(i)($m=1$) implies that
$$
P_{-1}=1,\ P_{-2}=P_{-3}=2,\ P_{-4}=P_{-5}=3,\ P_{-6}=4. 
$$
Then  $n_{1,2}^0=5$, $n_{1,3}^0=1$, $n_{1,4}^0=\epsilon_5=1-\sigma_5$, $0=\epsilon_6=1-\epsilon$.  Hence $\epsilon=1$, and this implies $\sigma_5=n_{1,5}^0=1$. Hence the basket $B^{(5)}=B^{(0)}=\{5\times (1,2),(1,3),(1,5)\}$ by $\epsilon_5=0$. Since $B^{(5)}$ admits no further prime packings, $B=B^{(5)}$ and $-K^3({\mathbb B})=-\frac{1}{30}<0$, a contradiction. Thus $\delta_1(X)\leq 6$.   

If $n_0=3$ and set $l=6$, then Theorem \ref{k2}(i)($m=1$) implies that
$$
P_{-1}=P_{-2}=1,\ P_{-3}=P_{-4}=P_{-5}=2,\ P_{-6}=3. 
$$
Then  $n_{1,2}^0=1$, $n_{1,3}^0=4$, $n_{1,4}^0=\epsilon_5=1-\sigma_5$, $0=\epsilon_6=1-\epsilon$.  Hence $\epsilon=1$, and this implies $\sigma_5=n_{1,5}^0=1$. Hence the basket $B^{(5)}=B^{(0)}=\{ (1,2),4\times(1,3),(1,5)\}$ by $\epsilon_5=0$. Since $B^{(5)}$ admits no further prime packings, $B=B^{(5)}$ and $-K^3({\mathbb B})=-\frac{1}{30}<0$, a contradiction. Thus $\delta_1(X)\leq 6$. 

If $n_0=4$ and set $l=6$, then Theorem \ref{k2}(i)($m=1$) implies that
$$
P_{-1}=P_{-2}=P_{-3}=1,P_{-4}=P_{-5}=P_{-6}=2. 
$$
Then  $n_{1,2}^0=2$, $n_{1,3}^0=1$, $n_{1,4}^0=3-\sigma_5$, $\epsilon_5=1-\sigma_5$, $0=\epsilon_6=1-\epsilon$.  Hence $\epsilon=1$, and this implies $\sigma_5=n_{1,5}^0=1$. Hence $B^{(5)}=\{2\times(1,2), (1,3), 2\times(1,4), (1,5)\}$ by $\epsilon_5=0$.  Hence $\epsilon_7\leq 1$ and $\epsilon_8=0$ by considering possible prime packings of $B^{(5)}$. On the other hand,  $\epsilon_7= P_{-7}-1$ and $\epsilon_8=P_{-8}-P_{-7}$. So $P_{-8}=\epsilon_7 +1 \leq 2$. But this contradicts $P_{-4}=2$ and inequality (\ref{2m}). So $\delta_1(X)\leq 6$. 

If $n_0=5$ and set $l=7$, then Theorem \ref{k2}(i)($m=1$) implies that
$$
P_{-1}=P_{-2}=P_{-3}=P_{-4}=1,\ P_{-5}=P_{-6}=P_{-7}=2. 
$$
Then  $n_{1,2}^0=2$, $n_{1,3}^0=2$, $n_{1,4}^0=2-\sigma_5$, $\epsilon_5=3-\sigma_5$, $0=\epsilon_6=2-\epsilon$, $\epsilon_7=1-2\sigma_5+2n_{1,5}^0+n_{1,6}^0$.  Hence $\epsilon=2$, and this implies $(\sigma_5,n_{1,5}^0)=(1, 0)$ or $(2,2)$. If $(\sigma_5,n_{1,5}^0)=(1, 0)$, then $n_{1,6}^0=1$ by $\epsilon_7\geq 0$. Hence $\epsilon_5=2$ and $B^{(5)}=\{ 2\times(2,5), (1,4),(1,6)\}$. Since $B^{(5)}$ admits no further prime packings, $B=B^{(5)}$ and $-K^3({\mathbb B})=-\frac{1}{60}<0$, a contradiction.  If $(\sigma_5,n_{1,5}^0)=(2, 2)$, then $\epsilon_5=1$, $\epsilon_7=1$, and $B^{(7)}=\{ (3,7), (1,3),2\times(1,5)\}$. Since $B^{(7)}$ admits no further prime packings, $B=B^{(7)}$ and $-K^3(\mathbb{B})=-\frac{2}{105}<0$, a contradiction. So $\delta_1(X)\leq 7$.   

If $n_0=6$ and set $l=8$, then Theorem \ref{k2}(i)($m=1$) implies that
$$
P_{-1}=P_{-2}=P_{-3}=P_{-4}=P_{-5}=1, \ P_{-6}=P_{-7}=P_{-8}=2. 
$$
Then  $n_{1,2}^0=2$, $n_{1,3}^0=2$, $n_{1,4}^0=2-\sigma_5$, $\epsilon_5=2-\sigma_5$, $0=\epsilon_6=3-\epsilon$.  Hence $\epsilon=3$ and $\sigma_5\leq 2$, and this implies $(\sigma_5,n_{1,5}^0)=(2,1)$. Then $\epsilon_5=0$ and $B^{(5)}=\{ 2\times(1,2), 2\times(1,3), (1,5),(1,s')\}$ for some $s'\geq 6$.  This implies $\epsilon_7=\epsilon_8=0$ since there are no further packings. On the other hand, $\epsilon_7=2-2\sigma_5+2n_{1,5}^0+n_{1,6}^0$ and $\epsilon_8=2-3\sigma_5+3n_{1,5}^0+2n_{1,6}^0+n_{1,7}^0$. Hence $n_{1,6}^0=0$, $n_{1,7}^0=1$, and $B^{(7)} =\{ 2\times(1,2), 2\times(1,3), (1,5),(1,7)\}$. Since $B^{(7)}$ is minimal, $B=B^{(7)}$ and $-K^3(\mathbb{B})=-\frac{1}{105}<0$, a contradiction. Thus $\delta_1(X)\leq 8$. 

If $n_0\geq 7$ and set $l=9$, then Theorem \ref{k2}(i)($m=1$) implies that
$$
P_{-1}=P_{-2}=P_{-3}=P_{-4}=P_{-5}=P_{-6}=1,\  P_{-8}=P_{-9}=2. 
$$
Then  $n_{1,2}^0=2$, $n_{1,3}^0=2$, $n_{1,4}^0=2-\sigma_5$, $\epsilon_5=2-\sigma_5$, $0=\epsilon_6=2-\epsilon$.  Hence $\epsilon=2$ and $\sigma_5\leq 2$, and this implies $(\sigma_5,n_{1,5}^0)=(1, 0)$ or $(2,2)$. If $(\sigma_5,n_{1,5}^0)=(2, 2)$, then $B^{(5)}=\{ 2\times(1,2), 2\times(1,3),2\times(1,5)\}$ by $\epsilon_5=0$.  Since $B^{(5)}$ admits no further prime packings, $B=B^{(5)}$ and $-K^3(\mathbb{B})<0$, a contradiction.  

Thus we are left to consider the case: $(\sigma_5,n_{1,5}^0)=(1, 0)$. Then we have $B^{(5)}=\{ (1,2), (2,5), (1,3), (1,4),(1,s')\}$ with $s' \ge 6$ by $\epsilon_5=1$.
Assume that $s'=6$, $7$. Clearly any basket $B$, with such a given $B^{(5)}$, dominates one of the following minimal ones: 
\begin{align*}
{}&B_1=\{(3,7), (2,7), (1,s')\};\\
{}&B_2=\{(1,2), (3,8), (1,4), (1,s')\}.
\end{align*}
Since $\sigma'(B)\geq \sigma'(B_i)\geq 2$ where $s'=6, 7$ and
$i=1,2$, inequality (\ref{vol}) fails for all $B$, which says that this
case does not happen. Hence $s' \ge 8$, then the expression of  $\epsilon_8$ gives
$$P_{-8}-P_{-7}=\epsilon_8+1.$$
Hence $P_{-7} = P_{-6} =1$ and $\epsilon_7=\epsilon_8=0$ since $P_{-8}=2$.
We have $B^{(8)}=B^{(5)}=\{ (1,2), (2,5), (1,3),
(1,4),(1,s')\}$ with $s'\geq 8$. Since $B^{(8)}$ admits no further prime packings, $B=B^{(8)}$.  By inequalities (\ref{kwmt}) and (\ref{vol}), $s'$ can only be $9,10,11$.
But then direct calculations show that $P_{-9}=3$ in all these three cases, a contradiction. We have proved $\delta_1(X)\leq 9$. 

So we conclude the theorem.
\end{proof}

\subsection{$\mathbb{Q}$-Fano $3$-folds with  $h^0(-K)=0$}\label{P=0}\

In this subsection we prove the following theorem.

\begin{thm}\label{p0}
Let $X$ be a $\mathbb{Q}$-Fano $3$-fold with $P_{-1}=0$. Then there exists an integer $m_1\leq 11$ such that $\dim\overline{\varphi_{-m_1}(X)}>1$.  Moreover, we can take such a number $m_1\leq 8$ except for the following baskets of singularities:
{\small $$\begin{array}{ll}
\rm{No.1.} &\{2 \times (1,2), 3 \times (2,5),(1,3),(1,4)\};\\
\rm{No.2.} &\{5 \times (1,2), 2 \times (1,3),(2,7),(1,4)\} ;\\
\rm{No.3.}&\{5 \times (1,2), 2 \times (1,3),(3,11)\} ;\\
\rm{No.4.}&  \{5 \times (1,2), (1,3),(3,10),(1,4)\};\\
\rm{No.A.}&  \{7\times (1,2), (3,7), (1,5) \};\\
\rm{No.B.}&  \{6\times (1,2), (4,9), (1,5) \};\\
\rm{No.C.}&  \{5\times (1,2), (5,11), (1,5) \};\\
\rm{No.D.}&  \{4\times (1,2), (6,13), (1,5) \};\\
\rm{No.E.}&  \{7\times (1,2), (3,8), (1,5) \};\\
\rm{No.F.}&  \{5\times (1,2), (4,9), (1,3), (1,5) \}.
\end{array}$$}
\end{thm}

\begin{remark}\label{optimal remark}
We do not know if this result is optimal since very few examples with $P_{-1}=0$ are known. There are $4$ known examples due to Iano-Fletcher \cite[List 16.7, No.60]{Fletcher} and Alt\i nok--Reid \cite{AR}, \cite[Example 9.14]{Reid}. For these examples we can see that  $\dim\overline{\varphi_{-8}(X)}>1$ by our theorem. Moreover, in next subsection we will treat the exceptional cases. If one can confirm either the existence or non-existence of type No.1--No.4, the result becomes optimal and so does Theorem \ref{main1}. 
\end{remark}

Before proving Theorem \ref{p0}, we recall a result by J. A. Chen and the first author.  

\begin{prop}[{\cite[Theorem 3.5]{CC}}]\label{list} 
Any geometric basket of weak $\bQ$-Fano 3-folds with $P_{-1}=P_{-2}=0$ is
among the following list:
 {\tiny
$$ \begin{array}{llccccccc}
&B & -K^3 & P_{-3} & P_{-4} & P_{-5} & P_{-6} & P_{-7} &P_{-8} \\
\hline\\
\rm{No.1.}& \{2 \times (1,2), 3 \times (2,5),(1,3),(1,4)\} & 1/60 & 0 & 0 & 1 & 1 & 1 &2\\
\rm{No.2.}& \{5 \times (1,2), 2 \times (1,3),(2,7),(1,4)\} & 1/84 & 0 & 1 & 0 & 1 & 1&2 \\
\rm{No.3.}& \{5 \times (1,2), 2 \times (1,3),(3,11)\} & 1/66 & 0 & 1 & 0 & 1 & 1 &2\\
\rm{No.4.}& \{5 \times (1,2), (1,3),(3,10),(1,4)\} & 1/60 & 0 & 1 & 0 & 1 & 1 &2\\
\rm{No.5.}& \{5 \times (1,2),  (1,3),2 \times (2,7)\} & 1/42 & 0 & 1 & 0 & 1 & 2&3 \\
\rm{No.6.}& \{4 \times (1,2),(2,5), 2 \times (1,3),2 \times (1,4)\} & 1/30 & 0 & 1 & 1 & 2 & 2 &4\\
\rm{No.7.}& \{3 \times (1,2),(2,5), 5 \times (1,3)\} & 1/30 & 1 & 1 & 1 & 3 & 3 &4\\
\rm{No.8.}& \{2 \times (1,2),(3,7), 5 \times (1,3)\} & 1/21 & 1 & 1 & 1 & 3 & 4 &5\\
\rm{No.9.}& \{(1,2),(4,9), 5 \times (1,3)\} & 1/18 & 1 & 1 & 1 & 3 & 4 &5\\
\rm{No.10.}& \{3 \times (1,2),(3,8), 4 \times (1,3)\} & 1/24 & 1 & 1 & 1 & 3 & 3&5 \\
\rm{No.11.}& \{3 \times (1,2),(4,11), 3 \times (1,3)\} & 1/22 & 1 & 1 & 1 & 3 & 3 &5\\
\rm{No.12.}& \{3 \times (1,2),(5,14), 2 \times (1,3)\} & 1/21 & 1 & 1 & 1 & 3 & 3 &5\\
\rm{No.13.}&\{2 \times (1,2),2 \times (2,5),4 \times (1,3)\} & 1/15 & 1 & 1 & 2 & 4 & 5 &7\\
\rm{No.14.}& \{ (1,2),(3,7),  (2,5),4 \times (1,3)\} & 17/210 & 1 & 1 & 2 & 4 & 6 &8\\
\rm{No.15.}& \{2 \times (1,2), (2,5),(3,8),3 \times (1,3)\} & 3/40 & 1 & 1 & 2 & 4 & 5&8 \\
\rm{No.16.}& \{2 \times (1,2),(5,13),3 \times (1,3)\} & 1/13 & 1 & 1 & 2 & 4 & 5 &8\\
\rm{No.17.}&\{ (1,2),3 \times (2,5),3 \times (1,3)\} & 1/10 & 1 & 1 & 3 & 5 & 7 &10\\
\rm{No.18.}& \{4 \times (1,2), 5 \times (1,3),(1,4)\} & 1/12 & 1 & 2 & 2 & 5 & 6 &9\\
\rm{No.19.}&\{4 \times (1,2), 4 \times (1,3),(2,7)\} & 2/21 & 1 & 2 & 2 & 5 & 7 &10\\
\rm{No.20.}& \{4 \times (1,2), 3 \times (1,3),(3,10)\} & 1/10 & 1 & 2 & 2 & 5 & 7 &10\\
\rm{No.21.}& \{3 \times (1,2),(2,5), 4\times (1,3),(1,4)\} & 7/60 & 1 & 2 & 3 & 6 & 8 &12\\
\rm{No.22.}& \{3 \times (1,2), 7 \times (1,3)\} & 1/6 & 2 & 3 & 4 & 9 & 12 &17 \\
\rm{No.23.}& \{2 \times (1,2),(2,5), 6 \times (1,3)\} & 1/5 & 2 & 3 & 5 & 10 & 14 &20\\
\end{array} $$}
\end{prop}

\begin{proof}[Proof of Theorem \ref{p0}]
In the proof, we will always take a suitable integer $m$ satisfying one of the conditions in Theorem \ref{k1}. If necessary, we apply Theorem \ref{k2} on $m$ and take $m_1=l_0m$.
\smallskip

{\bf Case I.} $P_{-2}=0$. 

The basket $B=B_X$ of the singularities of $X$ is among the list of Proposition \ref{list}. We just discuss it case by case. 

If $B$ is of type No.1, take $m=5$. Since $P_{-5}=1$ and $P_{-10}=4$, we can take $m_1=10$.

If $B$ is of type No.2, take $m=11$. Since  $P_{-11}=4$, we can take $m_1=11$.

If $B$ is of type No.3, take $m=10$. Since   $P_{-10}=3$, we can take $m_1=10$.

If $B$ is of type No.4, take $m=11$. Since   $P_{-11}=4$, we can take $m_1=11$.

If $B$ is of type No.5, take $m=8$. Since $P_{-8}=3$, we can take $m_1=8$.

If $B$ is of type No.6, take $m=8$. Since $P_{-8}=4$, we can take $m_1=8$.

If $B$ is of type No.7--No.21, take $m=3$. Since $P_{-3}=1$ and $P_{-6}\geq 3$, we can take $m_1=6$.

If $B$ is of type No.22--No.23, take $m=3$. Since $P_{-3}=2$ and $P_{-6}\geq 9$, we can take $m_1=6$.
\smallskip

{\bf Case II.} $P_{-2}>0$. 

Since  $P_{-1}=0$, the basket $B^{(0)}$ has datum
 \begin{numcases}{}
n_{1,2}^0=5 +4P_{-2}-P_{-3};\notag\\
n_{1,3}^0=4 -2P_{-2}+3P_{-3}-P_{-4};\notag\\
n_{1,4}^0=1 -P_{-2}-2P_{-3}+P_{-4}-\sigma_5.\notag
\end{numcases}
By Lemma \ref{31}, $B^{(0)}$ satisfies inequality (\ref{kwmt}) and thus
\begin{align*}
0{}&< \gamma(B^{(0)})= \sum_{r \ge 2}(\frac{1}{r}-r)
n^0_{1,r}+24\\
{}&\leq\sum_{r=2,3,4}(\frac{1}{r}-r)
n^0_{1,r}-\frac{24}{5}\sigma_5+24\\
{}&=\frac{25}{12}+\frac{37}{12}P_{-2}+P_{-3}-\frac{13}{12}P_{-4}-\frac{21}{20}\sigma_5.
\end{align*}
Hence,  by $n^0_{1,3}\geq 0$ and $n^0_{1,4}\geq 0$, we have
\begin{numcases}{}
\frac{25}{12}+\frac{37}{12}P_{-2}+P_{-3}-\frac{13}{12}P_{-4}-\frac{21}{20}\sigma_5>0;\label{1}\\
4 -2P_{-2}+3P_{-3}-P_{-4} \geq 0;\label{2}\\
1 -P_{-2}-2P_{-3}+P_{-4}-\sigma_5 \geq 0.\label{3}
\end{numcases}
Considering the inequality ``(\ref{1})+(\ref{2})+$2\times$(\ref{3})'':
\begin{align}\label{24}
\frac{97}{12}-\frac{11}{12}P_{-2}-\frac{1}{12}P_{-4}-\frac{61}{20}\sigma_5>0,
\end{align}
we obtain
$\sigma_5\leq 2.$
\smallskip

{\bf Subcase II-1.} $\sigma_5=0.$

At first, we consider the case $P_{-3}=0$. By inequality (\ref{2}), we have $2P_{-2}+ P_{-4} \leq 4$. Since $1\leq P_{-2}\leq P_{-4}$, it follows that $(P_{-2},P_{-4})=(1,1)$ or $(1,2)$. If $(P_{-2},P_{-4})=(1,1)$, then $B^{(0)}=\{9\times(1,2),(1,3),(1,4)\}$ with
$-K^3(B^{(0)}) =-\frac{1}{12}<0$. By considering a minimal basket
$B_{\min}$ dominated by $B^{(0)}$, then either
$B_{\min}=\{(10,21),(1,4)\}$ with $-K^3(B_{\min})=-\frac{1}{84}<0$ or
$B_{\min}=\{ 9 \times (1,2), (2,7)\}$ with $-K^3(B_{\min})=-\frac{1}{14}<0$.
Thus $-K^3(B)\leq -K^3(B_{\min})<0$, a contradiction.
If $(P_{-2},P_{-4})=(1,2)$, then $B^{(0)}=\{9 \times (1,2), 2\times (1,4)\}$. Since $B^{(0)}$ admits no prime packings anymore, $B=B^{(0)}$ and $-K^3(B)=0$, a contradiction.

Let us consider the case $P_{-3}\geq 1$. Since $\sigma_5=0$, $B^{(0)}$ is composed of $(1,2),(1,3),(1,4)$. In particular, $4b\geq r$ holds for every pair $(b, r)\in B^{(0)}$. As an easy conclusion, after packings, $4b\geq r$ holds for every pair $(b, r)\in B$. So $m=3$ satisfies the condition of Theorem \ref{k1}. By Theorem \ref{k2}, we can take $m_1=3$ or $6$ unless $(P_{-3},P_{-6})=(1,1),(1,2),(2,3)$. By  inequality  (\ref{3}),
\begin{align}\label{4323}
P_{-4}\geq 2P_{-3}+P_{-2}-1\geq 2P_{-3}.
\end{align}
By $P_{-2}>0$, $P_{-6}\geq  P_{-4}$. Thus we only need to consider the case $(P_{-3},P_{-6})= (1,2)$. By  inequality  (\ref{4323}), $P_{-2}=1$ and $P_{-4}=2$. On the other hand,
$$0=\epsilon_6=3P_{-1}+P_{-2}-P_{-3}-P_{-4}-P_{-5}+P_{-6}-\epsilon=-P_{-5}.$$
This implies $P_{-5}=0$ which contradicts $P_{-2}=P_{-3}=1$.
\smallskip

{\bf Subcase II-2.} $\sigma_5=2.$

By inequality (\ref{24}) and $P_{-4}\geq 2P_{-2}-1$, we have $P_{-2}\leq 1$. Hence $P_{-2}=1$ and, by inequalities (\ref{1})--(\ref{3}), we have inequalities:
 \begin{numcases}{}
\frac{46}{15} +P_{-3}-\frac{13}{12}P_{-4} >0;\label{4}\\
2+3P_{-3}-P_{-4} \geq 0;\label{5}\\
-2-2P_{-3}+P_{-4}  \geq 0.\label{6}
\end{numcases}
Considering the inequality ``$2\times (\ref{4})+(\ref{6})$'', we have $P_{-4}\leq 3$. Hence $P_{-3}=0$ by  inequality (\ref{6}), and $P_{-4}=2$ by inequalities  (\ref{5}) and (\ref{6}). Then $B^{(0)}=\{9 \times (1,2), (1,s_1), (1,s_2)\}$ with $5\leq s_1\leq s_2$. If $s_2>5$, then $\gamma(B^{(0)})\leq 9\times (\frac{1}{2}-2)+(\frac{1}{5}-5)+(\frac{1}{6}-6)+24<0$, a contradiction. Thus  $B^{(0)}=\{9 \times (1,2), 2\times(1,5) \}$. Since $B^{(0)}$ admits
no further prime packings, $B=B^{(0)}$. Take $m=5$. Since $P_{-5}=3$ by $\epsilon_5=0$, we can take $m_1=5$ by Theorem \ref{k2}.
\smallskip

{\bf Subcase II-3.} $\sigma_5=1.$

By inequalities (\ref{1})--(\ref{3}), we have
\begin{numcases}{}
12+37P_{-2}+12P_{-3}-13P_{-4}\geq0;\label{7}\\
4 -2P_{-2}+3P_{-3}-P_{-4} \geq 0;\label{8}\\
 -P_{-2}-2P_{-3}+P_{-4}\geq 0.\label{9}
\end{numcases}
Considering  the inequality  ``$(\ref{7})+13\times (\ref{9})$'', we have 
\begin{align}\label{3<2}
7P_{-3}\leq 12P_{-2}+6.
\end{align}
Considering the inequality ``$(\ref{8})+ (\ref{9})$'', we have 
\begin{align}\label{2<3}
3P_{-2}\leq  P_{-3}+4.
\end{align}
Inequalities (\ref{3<2}) and (\ref{2<3}) imply $P_{-2}\leq 3$.
\smallskip

{\bf Subsubcase II-3-i.} $(\sigma_5,P_{-2})=(1,3)$.

By  inequalities  (\ref{3<2}) and (\ref{2<3}), $5\leq P_{-3}\leq 6$.

If $P_{-3}=6$, by  inequalities (\ref{7}) and (\ref{9}), $P_{-4}=15$. Then $B^{(0)}=\{11\times (1,2),  (1,3), (1,s) \}$ for some integer $s\geq 5$.  By $\gamma(B^{(0)})>0$, we have $s=5$. Since the one-step packing $B_1=\{10\times (1,2),  (2,5), (1,5) \}$ has negative $\gamma(B_1)$, $B=B^{(0)}=\{11\times (1,2),  (1,3), (1,5) \}$. Take $m=4$. Since $P_{-4}=15$, we can take $m_1=4$ by Theorem \ref{k2}.

If $P_{-3}=5$, by inequalities (\ref{8}) and (\ref{9}), $P_{-4}=13$. Then $B^{(0)}=\{12\times (1,2),  (1,s) \}$ for some integer $s\geq 5$. By  $\gamma(B^{(0)})>0$, we have $s=5,6$. Clearly $B=B^{(0)}$.  Take $m=5$. Since $P_{-5}=22$, we can take $m_1=5$ by Theorem \ref{k2}.
\smallskip

{\bf Subsubcase II-3-ii.}  $(\sigma_5,P_{-2})=(1,2)$. 

By  inequalities  (\ref{3<2}) and (\ref{2<3}), $2\leq P_{-3}\leq 4$.

If $P_{-3}=4$, by  inequalities  (\ref{7}) and (\ref{9}), $P_{-4}=10$. Then $B^{(0)}=\{9\times (1,2), 2\times (1,3), (1,s) \}$ for some integer $s\geq 5$. By $\gamma(B^{(0)})>0$, we have $s=5$. Since the one-step packing  $B_1=\{8\times (1,2), (2,5), (1,3),(1,5) \}$ has negative $\gamma(B_1)$, $B=B^{(0)}=\{9\times (1,2),  2\times(1,3), (1,5) \}$. Take $m=4$. Since $P_{-4}=10$, we can take $m_1=4$ by Theorem \ref{k2}.

If $P_{-3}=3$, by inequalities (\ref{8}) and (\ref{9}), $8\leq P_{-4}\leq 9$. {}Firstly let us consider the case
$P_{-4}= 9$. Clearly $B^{(0)}=\{10\times (1,2), (1,4), (1,s) \}$ for some integer $s\geq 5$. By $\gamma(B^{(0)})>0$, we have $s=5$. If $B=B^{(0)}$, we may take $m=4$. Since $P_{-4}\geq 9$, we can take $m_1=4$ by Theorem \ref{k2}. If $B\neq B^{(0)}$, we have $B=\{10\times (1,2),(2,9) \}$.
Take $m=8$. Since $P_{-8}\geq 3$, we can take $m_1=8$ by Theorem \ref{k2}.
Now we consider the case $P_{-4}= 8$. We have $B^{(0)}=\{10\times (1,2), (1,3), (1,s) \}$ for some integer $s\geq 5$. Since $\gamma(B^{(0)})>0$, we have  $5\leq s\leq 6$. 
For the case $(P_{-4}, s)=(8,6)$, we get  $B =\{10\times (1,2), (1,3), (1,6) \}$ since any possible packing of $B^{(0)}$ has negative $\gamma$. Take $m=5$. Since $P_{-5}=13$, we can take $m_1=5$ by Theorem \ref{k2}.
For the case $(P_{-4}, s)=(8,5)$, we get either 
$B =\{10\times (1,2), (1,3), (1,5) \}$ or $B=\{9\times (1,2),   (2,5), (1,5) \}$ or $B=\{8\times (1,2),  (3,7), (1,5)\}$ by $\gamma>0$.
For all these cases, take $m=6$. Since $P_{-6}\geq 3$, we can take $m_1=6$ by Theorem \ref{k2}.

If $P_{-3}=2$, we have $P_{-4}=6$ by  inequalities  (\ref{8}) and (\ref{9}). Then $B^{(0)}=\{11\times (1,2), (1,s) \}$ for some integer $s\geq 5$. Similarly, $\gamma(B^{(0)})>0$ implies $5\leq s\leq 7$. Since $B^{(0)}$ admits
no further prime packings, $B=B^{(0)}$. Take $m=6$. Since $P_{-6}\geq 3$, we can take $m_1=6$ by Theorem \ref{k2}.
\smallskip

{\bf Subsubcase II-3-iii.} $(\sigma_5,P_{-2})=(1,1)$.

By inequality (\ref{3<2}), $P_{-3}\leq 2$.

If $P_{-3}=2$, we have $P_{-4}=5$ by inequalities (\ref{7}) and (\ref{9}). Then $B^{(0)}=\{7\times (1,2), 3\times(1,3),(1,s) \}$ for some integer $s\geq 5$. Similarly, $\gamma(B^{(0)})>0$ implies $s=5$. Furthermore, we have either 
$B =\{7\times (1,2), 3\times(1,3), (1,5) \}$ or $B=\{6\times (1,2),(2,5), 2\times(1,3), (1,5) \}$ by $\gamma>0$. 
Take $m=4$. Since $P_{-4}=5$, we can take $m_1=4$ by Theorem \ref{k2}.

If $P_{-3}=1$, we have $3\leq P_{-4}\leq 4$ by  inequalities  (\ref{7}) and (\ref{9}). Consider the case $(P_{-3},P_{-4})= (1,4)$. We  have $$B^{(0)}=\{8\times (1,2),(1,3), (1,4), (1,s) \}$$ for some integer $s\geq 5$. Again we have $s=5$ since $\gamma(B^{(0)})>0$. With the property $\gamma>0$ and considering all possible baskets with $B^{(0)}$, we see that $B$ must be one of the following baskets:
\begin{align*}
B_1={}&\{8\times (1,2),   (1,3), (1,4), (1,5) \},\\ 
B_2={}&\{8\times (1,2),   (2,7), (1,5)\},\\
B_3={}&\{8\times (1,2), (1,3), (2,9) \},\\
B_4={}&\{7\times (1,2),   (2,5), (1,4), (1,5)\}. 
\end{align*}
For  $B_2$, take $m=6$. Since $P_{-6}(B_2)\geq 3$, we can take $m_1=6$ by Theorem \ref{k2}.
For $B_3$, take $m=8$. Since $P_{-8}(B_3)\geq 3$, we can take $m_1=8$ by Theorem \ref{k2}.
For $B_1$ and $B_4$, take $m=4$. Similarly we can take $m_1=4$ by Theorem \ref{k2}. Consider the case $(P_{-3}, P_{-4})=(1,3)$. We have $B^{(0)}=\{8\times (1,2),   2\times (1,3), (1,s) \}$ for some integer $s\geq 5$. Similarly, 
$\gamma(B^{(0)})>0$ implies $5\leq s\leq 6$. 
If $s=6$, we see either $B =\{8\times (1,2),   2\times (1,3), (1,6) \}$ or $B=\{7\times (1,2), (2,5), (1,3), (1,6) \}$ since $\gamma(B)>0$. Take $m=5$. Since $P_{-5}\geq 3$, we can take $m_1=5$ by Theorem \ref{k2}.
If $s=5$, by considering all possible packings dominated by $B^{(0)}$ and using the property $\gamma>0$, we see that $B$ must be one of the following baskets:
\begin{align*}
B_i ={}&\{8\times (1,2),   2\times (1,3), (1,5)\},\\ 
B_{ii}={}&\{7\times (1,2), (2,5), (1,3), (1,5) \},\\
B_{iii}={}&\{6\times (1,2), 2\times (2,5), (1,5) \},\\
B_{iv}={}&\{6\times (1,2), (3,7), (1,3), (1,5) \},\\
B_{v}={}&\{5\times (1,2), (4,9), (1,3), (1,5) \},\\
B_{vi}={}&\{7\times (1,2), (3,8), (1,5) \},\\
B_{vii}={}&\{5\times (1,2), (3,7),  (2,5), (1,5) \}.
\end{align*}
For $B_{v}$ (corresponding to No.F) and $B_{vi}$ (corresponding to No.E), take $m=9$. Since $P_{-9}\geq 3$, we can take $m_1=9$ by Theorem \ref{k2}. For other cases, take $m=6$. Since $P_{-6}\geq 3$, we can take $m_1=6$ by Theorem \ref{k2}.

If $P_{-3}=0$, by  inequality  (\ref{8}), $P_{-4}\leq 2$. 
Firstly, consider the case $(P_{-3}, P_{-4})=(0,2)$. 
We have $B^{(0)}=\{9\times (1,2), (1,4), (1,s) \}$ for some integer $s\geq 5$. In fact, $5\leq s\leq 6$ by 
$\gamma(B^{(0)})>0$. When $s=6$, $B=B^{(0)}$ since $B^{(0)}$ admits no further packings. Take $m=7$. Since $P_{-7}=6$, we can take $m_1=7$ by Theorem \ref{k2}. When $s=5$, the property $\gamma >0$ implies that $B^{(0)}$ admits at most one further packings. Thus either $B =\{9\times (1,2),    (1,4), (1,5) \}$ (take $m=4$) or $B= \{9\times (1,2), (2,9) \}$ (take $m=8$). For the first basket, $P_{-4}=2$ and $P_{-8}=7$, we can take $m_1=8$ by Theorem \ref{k2}. For the second basket, $P_{-8}=7$ and we can take $m_1=8$ by Theorem \ref{k2}.

{}Finally we consider the case $(P_{-3}, P_{-4})=(0,1)$.  We have $B^{(0)}=\{9\times (1,2), (1,3), (1,s) \}$ for some integer $s\geq 5$. Similarly, $\gamma(B^{(0)})>0$ implies $5\leq s\leq 7$. When $s=7$, the property $\gamma>0$ implies that either $B=\{9\times (1,2), (1,3), (1,7) \}$ or $B=\{8\times (1,2), (2,5), (1,7) \}$.  Take $m=8$. Since $P_{-8}\geq 3$, we can take $m_1=8$ by Theorem \ref{k2}.
When $s=6$, the inequalities $\gamma>0$ and $-K^3>0$ imply that $B$ must be one of the following baskets:
\begin{align*}
{}&\{8\times (1,2), (2,5), (1,6) \},\\ 
{}&\{7\times (1,2), (3,7), (1,6) \},\\
{}&\{6\times (1,2), (4,9), (1,6) \}.
\end{align*}
Take $m=7$. Since $P_{-7}\geq 3$, we can take $m_1=7$ by Theorem \ref{k2}.
When $s=5$, inequalities $\gamma >0$ and $-K^3>0$ imply that $B$ is among one of the following baskets: 
\begin{align*}
B_a={}&\{7\times (1,2), (3,7), (1,5) \},\\ 
B_b={}&\{6\times (1,2), (4,9), (1,5) \},\\
B_c={}&\{5\times (1,2), (5,11), (1,5) \},\\
B_d={}&\{4\times (1,2), (6,13), (1,5) \}.
\end{align*}
For $B_d$ (corresponding to No.D), take $m=11$. Since $P_{-11}\geq 3$, we can take $m_1=11$ by Theorem \ref{k2}.
For other baskets (corresponding to No.A--No.C), take $m=9$. Since $P_{-9}\geq 3$, we can take $m_1=9$ by Theorem \ref{k2}. So the theorem is proved. 
\end{proof}

\subsection{Exceptional cases}\label{exceptional}\

In this subsection, we treat the exceptional cases in Theorem \ref{p0}.

\begin{thm}\label{ex}
Let $X$ be a $\mathbb{Q}$-Fano $3$-fold with basket of singularities $B$.
\begin{itemize}
\item[(i)]  If $B$ is of type No.1--No.4 as in Theorem \ref{p0}, then  $\dim\overline{\varphi_{-10}(X)}>1$. 
\item[(ii)]  If $B$ is of type No.A--No.D as in Theorem \ref{p0}, then  $\dim\overline{\varphi_{-8}(X)}>1$. 
\item[(iii)]  If $B$ is of type No.E--No.F as in Theorem \ref{p0}, then  $\dim\overline{\varphi_{-6}(X)}>1$.
\end{itemize}
\end{thm}
\begin{proof} (i).  Recall the proof in Case I of Theorem \ref{p0}. We may only consider the two cases with No.2 and No.4. Since $P_{-9}=2$, $\delta_1(X)\geq 10$.  We want to show that $\delta_1(X)=10$ in both cases. In fact, we have $P_{-4}= P_{-6}=1$, $P_{-8}=2$, $P_{-10}\geq 3$.  Note that the conditions of Theorem \ref{k1} are all satisfied with $m=4$.  It follows that $-4K_X\sim E$ is a prime divisor.  Assume that $\dim\overline{\varphi_{-10}(X)}=1$, then we can write $|-10K_X|=|nS|+E'$ with $n\geq 2$, $|S|$ is an irreducible rational pencil of surfaces and $E'$ is the fixed part. By $P_{-6}>0$, we have $ E\leq |nS|+E'$. Since $E$ is reduced and irreducible, either $ E\leq |S| $ or $ E\leq  E'$ holds.  
Then 
$$
P_{-6}=h^0(-10K_X-E)=h^0(nS+E'-E)\geq h^0(S)=2,
$$
a contradiction.

(ii). Recall the last part of Subsubcase II-3-iii in the proof of Theorem \ref{p0}. If $B$ is of type No.A--No.D, we have $P_{-2}= P_{-4}=1$, $P_{-6}=2$, and $P_{-8}= 3$. Assume, to the contrary,  that  $\dim\overline{\varphi_{-8}(X)}=1$. 

Write $-2K_X\sim D$ for some effective Weil divisor. By Theorem \ref{k2}(i) (with $m=2$), $D$ must be either reducible or non-reduced. As in the proof of Theorem \ref{k1}, take $E$ to be any strictly effective divisor such that $E<D$.  Then   inequality  (\ref{qq}) must fail for some singularity $Q$ in $B_a$--$B_d$. Clearly, such an offending singularity $Q$ must be ``$(1,5)$''. By equality (\ref{qq2}), the local index $i_Q(E)$ of $E$ should be $4$ since  inequality  (\ref{qq}) holds for $i\in\{0,1,2,3\}$ and $(b,r)=(1,5)$, that is, $E\sim -K_X$ at $Q$. Since $E$ is arbitrary such that $0<E<D$ and $i_Q(-2K_X)=3$, we conclude that $D=E_1+E_2$ where $E_i$ is fixed prime divisor with $i_Q(E_i)=4$ for $i=1,2$. 

If $E_1=E_2$, then $2(-K_X-E_1)\sim 0$.  By \cite[Proposition 2.9]{Prok} and  the fact that $-K_X-E_1$ is Cartier at $Q$, we conclude that $-K_X-E_1$ is not $2$-torsion. Hence $-K_X-E_1\sim 0$, which contradicts $P_{-1}=0$. Thus $E_1$ and $E_2$ are different prime divisors. 

Since $|-6K_X|\preceq |-8K_X|$, by Lemma \ref{pencils} we can write
\begin{align*}
|-6K_X|{}&=|S|+a_6E_1+b_6E_2,\\
|-8K_X|{}&=|2S|+a_8E_1+b_8E_2,
\end{align*}
where $|S|$ is an irreducible rational pencil of surfaces, $a_iE_1+b_iE_2$ is the fixed part, $a_i,b_i\in \mathbb{N}$ for $i=6,8$.

\begin{claim}
$a_6b_6=a_8b_8=0$.
\end{claim}
\begin{proof}
Assume that $a_6,b_6\geq 1$, then
$$
P_{-4}=h^0(-6K_X-E_1-E_2)\geq h^0(S)=2,
$$
a contradiction. Similarly, we have $a_8b_8=0$.
\end{proof}

We may assume that $b_6=0$. Then 
\begin{align}\label{a6}
3E_1+3E_2\in |S+a_6E_1|=|S|+a_6E_1.
\end{align}
It follows that $a_6\leq 3$.
\smallskip

{\bf Case ii.1.} $b_8=0$.

In this case
$$
2S+a_8E_1 \sim -8K_X \sim -6K_X+E_1+E_2 \sim S+(a_6+1)E_1+E_2.
$$
Since $a_8E_1$ is the fixed part of $|2S+a_8E_1|$, $a_8\leq a_6+1$.
Then
\begin{align}\label{x2}
S  \sim (a_6+1-a_8)E_1+E_2.
\end{align}
By relations (\ref{a6}) and (\ref{x2}),
\begin{align}\label{xa}
(2a_6+1-a_8)E_1+E_2\sim 3E_1+3E_2.
\end{align}
Clearly, $2a_6+1-a_8\leq 3$ is absurd. Thus $2a_6+1-a_8\geq 4$. On the other hand $2a_6+1-a_8\leq 7$ since $a_6\leq 3$.
Locally at $Q$, since $i_Q(E_1)=i_Q(E_2)=4$, we have
$$
2a_6+1-a_8\equiv 0 \mod 5.
$$
So $2a_6+1-a_8=5$. Then relation (\ref{xa}) implies $2E_1\sim 2E_2$. By \cite[Proposition 2.9]{Prok}, we conclude that $E_1\sim E_2$, a contradiction.
\smallskip

{\bf Case ii.2.} $a_8=0$ and $b_8>0$.

In this case
$$
2S+b_8E_2 \sim -8K_X \sim -6K_X+E_1+E_2 \sim  S+(a_6+1)E_1+E_2.
$$
This implies that $b_8\leq 1$. Hence $b_8=1$ and 
\begin{align}\label{x3}
S \sim  (a_6+1)E_1.
\end{align}
By relations (\ref{a6}) and (\ref{x3}),
\begin{align}
(2a_6+1)E_1 \sim 3E_1+3E_2.
\end{align}
Clearly $2a_6+1\geq 4$ and $2a_6+1\leq 7$ since $a_6\leq 3$.
Locally at $Q$, since $i_Q(E_1)=i_Q(E_2)=4$, we have
$$
2a_6+1\equiv 1 \mod 5.
$$
Since $4\leq 2a_6+1\leq 7$, this is impossible.
\smallskip

(iii).  Recall the cases with $B_v$ (No.F) and $B_{vi}$ (No. E) (see Subsubcase II-3-iii in the proof of Theorem \ref{p0}). We have $P_{-2}= 1$, $P_{-4}=3$, $P_{-6}=9$. Assume, to the contrary, that  $\dim\overline{\varphi_{-6}(X)}=1$. 

We can write $-2K_X\sim D$ for some effective divisor $D$. By the same argument as (ii), $D=E_1+E_2$ with $E_i$ reduced and irreducible and $i_Q(E_i)=4$ for $i=1,2$ where $Q$ is the singularity ``$(1,5)$''.  Note that, however, we do not know if $E_1$ and $E_2$ are different.

Since $|-4K_X|\preceq |-6K_X|$, by Lemma \ref{pencils} we can write
\begin{align*}
|-4K_X|{}&=|2S|+a_4E_1+b_4E_2,\\
|-6K_X|{}&=|8S|+a_6E_1+b_6E_2,
\end{align*}
where $|S|$ is an irreducible  rational pencil of surfaces and $a_iE_1+b_iE_2$ is the fixed part, $a_i,b_i\in \mathbb{N}$ for $i=4,6$.
Hence
$$
2S+a_4E_1+b_4E_2\sim -4K_X\sim2(-2K_X)\sim 2E_1+2E_2.
$$
Since $a_4E_1+b_4E_2$ is the fixed part of $| 2S+a_4E_1+b_4E_2|$,  we may assume $a_4\leq b_4\leq 2$.

If $b_4=2$, then $2S\sim (2-a_4)E_1$. Hence $E_1\leq S$ by the irreducibility of $E_1$.  Then 
$$
1=h^0(E_1 )\geq h^0(2S-E_1)\geq h^0(S)=2,
$$
a contradiction.

If $b_4=1\geq a_4$, then $2S\sim (2-a_4)E_1+E_2\geq E_1+E_2.$ Hence $E_1\leq S$ by the irreducibility of $E_1$. Then 
$$
1=h^0(E_1 +E_2)\geq h^0(2S-E_1)\geq h^0(S)=2,
$$
a contradiction.

Hence $a_4=b_4=0$ and $
2S\sim 2E_1+2E_2.
$ Then
\begin{align*}
0\sim{}& -6K_X-3E_1-3E_2\\
\sim {}&4(2E_1+2E_2) +a_6E_1+b_6E_2-3E_1-3E_2\\
\geq{}& 5E_1+5E_2,
\end{align*}
a contradiction. So we have proved the theorem.
\end{proof}

To make the summary, Theorems \ref{p0}  and \ref{ex} directly imply the following:

\begin{cor}\label{00} Let $X$ be a $\mathbb{Q}$-Fano $3$-fold with $P_{-1}=0$. Then $\delta_1(X)\leq 8$ except for the following cases:
{\small $$\begin{array}{lll}
\rm{No.1.} &\{2 \times (1,2), 3 \times (2,5),(1,3),(1,4)\}&\delta_1(X)= 10;\\
\rm{No.2.} &\{5 \times (1,2), 2 \times (1,3),(2,7),(1,4)\}&\delta_1(X)=10;\\
\rm{No.3.}&\{5 \times (1,2), 2 \times (1,3),(3,11)\} &\delta_1(X)=10;\\
\rm{No.4.}&  \{5 \times (1,2), (1,3),(3,10),(1,4)\}&\delta_1(X)=10;\\
\rm{No.A.}&  \{7\times (1,2), (3,7), (1,5) \}&\delta_1(X)= 8;\\
\rm{No.B.}&  \{6\times (1,2), (4,9), (1,5) \}&\delta_1(X)=8;\\
\rm{No.C.}&  \{5\times (1,2), (5,11), (1,5) \}&\delta_1(X)=8;\\
\rm{No.D.}&  \{4\times (1,2), (6,13), (1,5) \}&\delta_1(X)=8;\\
\rm{No.E.}&  \{7\times (1,2), (3,8), (1,5) \}&\delta_1(X)\leq 6;\\
\rm{No.F.}&  \{5\times (1,2), (4,9), (1,3), (1,5) \}&\delta_1(X)\leq 6.
\end{array}$$}
\end{cor}

Theorem \ref{main1} follows directly from Theorems \ref{p3},  \ref{p2}, and \ref{p1},  and Corollary \ref{00}.

\section{\bf When is $|-mK_X|$ not composed with a pencil?  (Part II)}\label{section non-pencil weak}

As we have seen in the last section, the condition $\rho(X)=1$ is crucial to proving Theorem \ref{k1}.  For arbitrary weak $\bQ$-Fano 3-folds, we have to study in an alternative way. Naturally what we can prove is weaker than Theorem \ref{main1}.

 Let $X$ be a weak $\bQ$-Fano 3-fold.  We are going to estimate $\delta_1(X)$ from above. 
The main idea is to relate this problem to the value distribution of  the Hilbert function $\chi_{-m}=P_{-m}$. 

\begin{lem}\label{b L^2} Keep the same notation as in Subsection \ref{b setting}.
The number $r_X(\pi^*(-K_X)^2\cdot S)_Y$ is a positive integer.
\end{lem}
\begin{proof} Note that the number $(\pi^*(-K_X)^2\cdot S)_Y$ is
positive since $$(\pi^*(-K_X)^2\cdot S)_Y=(\pi^*(-K_X)|_S)_S^2$$ and
$\pi^*(-K_X)|_S$ is nef and big on $S$. It is independent of
the choice of $\pi$ according to the projection formula of the
intersection theory. So we may choose such a modification $\pi$ that
dominates a resolution of singularities $\tau: \hat{W}\rightarrow
X$. Then we see $(\pi^*(-K_X)^2\cdot S)_Y=(\tau^*(-K_X)^2\cdot
{S_1})_{\hat{W}}$ where $S_1=\theta_*(S)$ is a divisor on $\hat{W}$
and $\theta:Y\rightarrow \hat{W}$ is a birational morphism. Note that,
however, $S_1$ is a generic element in an algebraic family though it
is not necessarily nonsingular.

We may write $K_{\hat{W}}=\tau^*(K_X)+\Delta_{\tau}$ where
$\Delta_{\tau}$ is an exceptional effective $\bQ$-divisor over those
isolated terminal singularities on $X$. Now, by intersection theory,
we have $$(r_X\tau^*(-K_X)\cdot \tau^*(-K_X)\cdot
S_1)_{\hat{W}}=(r_X\tau^*(-K_X)\cdot (-K_{\hat{W}})\cdot S_1)_{\hat{W}}$$ is an
integer.
\end{proof}

\begin{cor}\label{b non-pencil} Let $X$ be a weak $\bQ$-Fano 3-fold.
If $$P_{-m}>r_X(-K_X)^3m+1$$ for some integer $m$, then
$|-mK_X|$ is not composed with a pencil.
\end{cor}
\begin{proof} Assume that $|-mK_X|$ is composed with a pencil. Set $D:=-mK_X$ and keep the same notation as in Subsection \ref{b setting}.
Then we have $m\pi^*(-K_X)\geq M_{-m}\equiv (P_{-m}-1)S$. Thus $$m(-K_X)^3\geq
(P_{-m}-1)(\pi^*(-K_X)^2\cdot S)\geq \frac{1}{r_X}(P_{-m}-1)$$ by
Lemma \ref{b L^2}, a contradiction. \end{proof}

Next we estimate the number $m$ which satisfies Corollary \ref{b non-pencil}. We will do this in two steps as follows. 
\begin{prop}\label{b thm1}
Let $X$ be a weak $\mathbb{Q}$-Fano $3$-fold. Take an arbitrary real number $0<t\leq 37$. Denote $r_{\max}:=\max\{r_i\in B_X\}$ the maximum of local indices of singularities. If $$m\geq \max\left\{37, \frac{r_{\max}t}{3}, \sqrt{6r_X+\frac{12}{t(-K_X^3)}}\right\},$$ then $P_{-m}\geq r_X (-K_X^3)m+2$. In particular, $|-mK_X|$ is not composed with a pencil.
\end{prop}

\begin{proof}
By Reid's formula, there exists a basket of singularities 
$$
B_X=\{(b_i,r_i)\mid i=1,\ldots, s; 0<b_i\leq \frac{r_i}{2}; b_i \mbox{ is coprime to } r_i\}
$$
such that we have the formula 
$$
P_{-n}=\frac{1}{12}n(n+1)(2n+1)(-K_X^3)+2n+1-l(-n)
$$
for any $n>0$, where
$$
l(-n)=\sum_i \sum_{j=1}^n \frac{\overline{jb_i}(r_i-\overline{jb_i})}{2r_i}. 
$$
To estimate the lower bound of $P_{-n}$, we need to bound $l(-n)$ from above. 

For any pair $(b,r)\in B_X$, we have $r\leq 24$ by inequality (\ref{kwmt2}).  In fact, we have the following estimation.
\begin{enumerate}
\item If $r=2$, then $$\frac{\overline{jb}(r-\overline{jb})}{2r}=
\begin{cases}
\frac{1}{4} & \mbox{when}\ j \mbox{ odd};\\
0 & \mbox{when}\  j \mbox{ even}.
\end{cases}
$$

\item If $r$ is odd, then $\frac{\overline{jb}(r-\overline{jb})}{2r}\leq \frac{r^2-1}{8r}$.

\item If $r$ is even and $r>2$, then $$\frac{\overline{jb}(r-\overline{jb})}{2r}\leq \begin{cases}
 \frac{\frac{r-2}{2}\frac{r+2}{2}}{2r}=\frac{r^2-4}{8r} & \mbox{when } \ \overline{jb}\neq r/2;\\
\frac{r^2}{8r} & \mbox{when }  \overline{jb}= r/2.
\end{cases}
$$
Clearly, $b \neq r/2$ under the same situation. Since $\overline{jb}= r/2$ and $\overline{(j-1)b}= r/2$ can not hold simultaneously, we have
$$\frac{\overline{(j-1)b}(r-\overline{(j-1)b})}{2r}+\frac{\overline{jb}(r-\overline{jb})}{2r}\leq \frac{r^2-4}{8r}+ \frac{r^2}{8r} \leq  \frac{2\cdot(r^2-1)}{8r}.$$
Hence, when $r$ is even and $r>2$,  we have
\begin{equation}
\sum_{j=1}^{n} \frac{\overline{jb}(r-\overline{jb})}{2r}\leq n \cdot \frac{r^2-1}{8r}.
\label{summ}\end{equation}
\end{enumerate}
By the way, inequality (\ref{summ}) also holds  when $r$ is odd.

Recall that we have 
$$
\sum_{j=1}^{r} \frac{\overline{jb}(r-\overline{jb})}{2r}=\frac{r^2-1}{12}.
$$
Hence, whenever $r>2$ and $n\geq \frac{r_{\text{max}}t}{3}$, we have
\begin{align}
\sum_{j=1}^n \frac{\overline{jb}(r-\overline{jb})}{2r}
{}&=\Big\lfloor \frac{n}{r}\Big\rfloor\frac{r^2-1}{12}+ \sum_{j=1}^{\overline{n}} \frac{\overline{jb}(r-\overline{jb})}{2r}\notag\\
{}&\leq \Big\lfloor \frac{n}{r}\Big\rfloor\frac{r^2-1}{12}+ {\rm min}\Big\{ \overline{n} \cdot\frac{r^2-1}{8r}, \frac{r^2-1}{12}\Big\}\notag\\
{}&\leq \frac{r^2-1}{12r}\big(n+\frac{r}{3}\big)\label{n+r3}\\
{}&\leq \frac{r^2-1}{12r}\cdot \frac{(t+1)n}{t}.\notag
\end{align}
We prove the second inequality here. Assume, to the contrary, that  
\begin{align}\label{b r31}
\Big\lfloor \frac{n}{r}\Big\rfloor\frac{r^2-1}{12}+ \overline{n} \cdot \frac{r^2-1}{8r} >\frac{r^2-1}{12r}\big(n+\frac{r}{3}\big),
\end{align}
and
\begin{align}\label{b r32}
\Big\lfloor \frac{n}{r}\Big\rfloor\frac{r^2-1}{12}+ \frac{r^2-1}{12}>\frac{r^2-1}{12r}\big(n+\frac{r}{3}\big).
\end{align}
Inequality (\ref{b r31}) implies
$
\overline{n}>\frac{2r}{3}.
$
But from inequality (\ref{b r32}), we have 
$
\overline{n}<\frac{2r}{3},
$
a contradiction.

Since $X$ is weak $\bQ$-Fano, recall that we have inequality 
\begin{align*}
\sum_i\big(r_i-\frac{1}{r_i}\big)\leq 24
\end{align*}
by inequality (\ref{kwmt2}). 
Denote by $N_2$ the number of $r_i=2$ in $B_X$. 
Then, if $n \geq \frac{r_{\text{max}}t}{3}$,
\begin{align*}
l(-n){}&=\sum_i \sum_{j=1}^n \frac{\overline{jb_i}(r_i-\overline{jb_i})}{2r_i}\\
{}&=\frac{N_2}{4}\Big\lfloor\frac{n+1}{2}\Big\rfloor+ \sum_{r_i>2} \sum_{j=1}^n \frac{\overline{jb_i}(r_i-\overline{jb_i})}{2r_i}\\
{}&\leq \frac{N_2}{4}\Big\lfloor\frac{n+1}{2}\Big\rfloor+ \frac{(t+1)n}{t}\sum_{r_i>2} \frac{r_i^2-1}{12r_i}\\
{}&\leq \frac{N_2}{4}\Big\lfloor\frac{n+1}{2}\Big\rfloor+ \frac{(t+1)n}{t}\cdot \frac{24-\frac{3}{2}N_2}{12}\\
{}&\leq \frac{2(t+1)n}{t}-N_2\Big(\frac{(t+1)n}{8t}-\frac{1}{4}\Big\lfloor\frac{n+1}{2}\Big\rfloor\Big)\\
{}& \leq \frac{2(t+1)n}{t}
\end{align*}
where $\frac{(t+1)n}{8t}-\frac{1}{4}\lfloor\frac{n+1}{2}\rfloor\geq 0$ whenever $n\geq t$. 
Hence 
\begin{align*}
P_{-n}={}&\frac{1}{12}n(n+1)(2n+1)(-K_X^3)+2n+1-l(-n)\\
\geq{}& \frac{1}{6}n^3(-K_X^3)+\frac{n^2}{4}(-K_X^3)+1-\frac{2n}{t}.
\end{align*}
By \cite{CC}, $-K_X^3\geq \frac{1}{330}$. Hence $\frac{n^2}{4}(-K_X^3)\geq 1$ if $n\geq 37$.
If $m\geq \sqrt{6r_X+\frac{12}{t(-K_X^3)}}$, then 
\begin{align*}
P_{-m}\geq {}&\frac{1}{6}m^3(-K_X^3)+2-\frac{2m}{t}\\
\geq{}&\frac{1}{6}\Big(6r_X+\frac{12}{t(-K_X^3)}\Big)m(-K_X^3)+2-\frac{2m}{t}\\
={}&r_X (-K_X^3)m+2.
\end{align*}
We complete the proof. 
\end{proof}

In practice, we will take a suitable $t$ to apply Proposition \ref{b thm1}. Note that $r_{\text{max}}\leq 24$.

\begin{prop}\label{b thm2}
Let $X$ be a weak $\mathbb{Q}$-Fano $3$-fold. 
\begin{itemize}
\item[(i)] If $r_X\leq 660$, then $\sqrt{6r_X+\frac{3}{2(-K_X^3)}}<67$. In particular, $P_{-m}\geq r_X (-K_X^3)m+2$
 for $m\geq 67$.
\item[(ii)] If $r_X>660$, then $r_X=840$, and $P_{-m}\geq r_X (-K_X^3)m+2$ for $m\geq 71$.
\end{itemize}
\end{prop}

\begin{proof}
Statement (i) is clear since $-K_X^3\geq \frac{1}{330}$ by \cite{CC} and take $t=8$ in Proposition \ref{b thm1}.
We mainly prove (ii) here. 

{}First of all, by Proposition \ref{bound index}, $r_X=840$ and $\mathcal{R}=(3,5,7,8)$ or $(2,3,5,7,8)$.

For $r>2$, we use the inequality (\ref{n+r3}) (in the proof of Proposition \ref{b thm1}) that
$$
\sum_{j=1}^n \frac{\overline{jb}(r-\overline{jb})}{2r}\leq \frac{r^2-1}{12r}\big(n+\frac{r}{3}\big).
$$
Then
\begin{align*}
l(-n)={}&\sum_i \sum_{j=1}^n \frac{\overline{jb_i}(r_i-\overline{jb_i})}{2r_i}\\
\leq {}&\frac{N_2}{4}\Big\lfloor\frac{n+1}{2}\Big\rfloor+ \sum_{r_i>2} \frac{r_i^2-1}{12r_i}\big(n+\frac{r_i}{3}\big)\\
\leq {}&\frac{n+1}{8}+ \frac{3^2-1}{12\cdot 3}(n+1)+ \frac{5^2-1}{12\cdot 5}\big(n+\frac{5}{3}\big)\\
{}&+\frac{7^2-1}{12\cdot 7}\big(n+\frac{7}{3}\big)+
\frac{8^2-1}{12\cdot 8}\big(n+\frac{8}{3}\big)\\
={}&\frac{19907n}{10080}+\frac{295}{72}\\
\leq {}&2n+\frac{7}{3}
\end{align*}
as long as $n\geq 71$.

Hence 
\begin{align*}
P_{-n}={}&\frac{1}{12}n(n+1)(2n+1)(-K_X^3)+2n+1-l(-n)\\
\geq{}& \frac{1}{6}n^3(-K_X^3)+\big(\frac{n^2}{4}(-K_X^3)-\frac{10}{3}\big)+2.
\end{align*}
By \cite{CC}, $-K_X^3\geq \frac{1}{330}$. Hence $\frac{n^2}{4}(-K_X^3)\geq \frac{10}{3}$ whenever  $n\geq 71$.
If $m\geq 71 > \sqrt{6r_X}$, then
\begin{align*}
P_{-m}\geq{}& \frac{1}{6}m^3(-K_X^3)+2\\
\geq{}&\frac{1}{6}(6r_X)m(-K_X^3)+2\\
={}&r_X (-K_X^3)m+2.
\end{align*}
We finish the proof. 
\end{proof}

Theorem \ref{main2} directly follows from Corollary \ref{b non-pencil} and Proposition \ref{b thm2}. 

\section{\bf Birationality}\label{section birationality}

In this section, we consider the birationality of anti-pluricanonical maps $\varphi_{-m}$.

\subsection{Main reduction}\
 
In this subsection, we reduce the birationality problem on $X$ to that on $Y$.

\begin{lem}[{cf. \cite[Lemma 2.5]{C}}]\label{b Hn} Let $W$ be a normal projective variety on which there is
an integral Weil $\bQ$-Cartier divisor $D$. Let $h: V\longrightarrow
W$ be any resolution of singularities. Assume that $E$ is an
effective exceptional $\bQ$-divisor on $V$ with $h^*(D)+E$ a Cartier
divisor on $V$. Then
$$h_*\OO_V(h^*(D)+E)=\OO_W(D)$$
where $\OO_W(D)$ is the reflexive sheaf corresponding to the Weil
divisor $D$.
\end{lem}

\begin{lem}[{cf. \cite[2.6]{C}}]\label{b reduction}
Let $X$ be a weak $\bQ$-Fano 3-fold and 
$\pi:Y\longrightarrow X$ the same resolution as in Subsection \ref{b setting}. Then, for any $m>0$,  
$\varphi_{-m}$ is birational if and only if so is
$\Phi_{|K_Y+\roundup{(m+1)\pi^*(-K_X)}|}$.
\end{lem}
\begin{proof} Recall that
$$K_{Y}=\pi^*(K_X)+E_{\pi}$$ 
where $E_{\pi}$ is an effective $\bQ$-Cartier $\bQ$-divisor since $X$ has at worst terminal singularities.
We have
\begin{align*}
{}& K_Y+\roundup{(m+1)\pi^*(-K_X)}\\
={}& \pi^*(K_X)+ E_{\pi}+\pi^*(-(m+1)K_X)+E_{m+1}\\
={}& \pi^*(-mK_X)+ E_{\pi}+E_{m+1}
\end{align*}
where $E_{\pi}+E_{m+1}$ is an effective $\bQ$-divisor on $Y$ exceptional over $X$. Lemma \ref{b Hn}
implies $$\pi_*\OO_Y(K_Y+\roundup{(m+1)\pi^*(-K_X)})=\OO_X(-mK_X).$$
Hence $\varphi_{-m}$ is birational if and only if so is
$\Phi_{|K_Y+\roundup{(m+1)\pi^*(-K_X)}|}$.
\end{proof}
Noting that
\begin{align*}
H^0(\OO_X(-mK_X))\cong{}& H^0(\OO_Y(\rounddown{-m\pi^*(K_X)}))\\
\cong {}&H^0(\OO_Y(K_Y+\roundup{(m+1)\pi^*(-K_X)})), 
\end{align*}
we denote by $|M_{-m}|$ the movable part of
$|\rounddown{-m\pi^*(K_X)}|$. We have the equality:
\begin{align}\label{b 2.1}
-m\pi^*(K_X)=M_{-m}+F_{m} 
\end{align}
where $F_m$ is an effective $\bQ$-divisor. 
Another direct consequence is that we may write:
$$K_Y+\roundup{(m+1)\pi^*(-K_X)}\sim M_{-m}+N_{-m}$$ where $N_{-m}$
is the fixed part of $|K_Y+\roundup{(m+1)\pi^*(-K_X)}|$.

\subsection{Key theorem}\label{b keythm}\ 

Let $X$ be a weak $\bQ$-Fano 3-fold on which $P_{-m_0}\geq
2$ for some integer $m_0>0$.  Suppose that $m_1\geq m_0$ is another integer with $P_{-m_1}\geq 2$ and that $|-m_1K_X|$ and $|-m_0K_X|$ are not composed with the same pencil.  
Recall that, for any $m>0$ with $P_{-m}>1$, 
$$\iota(m)=\begin{cases} 1, & \text{if } |-mK_X| \text{ is not composed with a pencil}; \\
P_{-m}-1, & \text{if }|-mK_X| \text{ is composed with a pencil}.
\end{cases}$$

Set $D:=-m_0K_X$ and keep the same notation as in Subsection \ref{b setting}. We may modify the resolution $\pi$ in Subsection \ref{b setting} such
that the movable part $|M_{-m}|$ of $|\rounddown{\pi^*(-mK_X)}|$ is base point free for all  $m_0\leq m\leq m_1$. 
Pick a generic irreducible element $S$ of $|M_{-m_0}|$.  
By equality (\ref{b 2.1}), we have 
$$m_0\pi^*(-K_X)=\iota(m_0)S+F_{m_0}$$ for some effective
$\bQ$-divisor $F_{m_0}$. In particular, we see that 
$$\frac{m_0}{\iota(m_0)}\pi^*(-K_X)-S\sim_{\bQ}  \text{effective}\ \bQ\text{-divisor}.$$
Define the real number 
$$\mu_0=\mu_0(|S|):=\text{inf}\{t\in \bQ^+ \mid t\pi^*(-K_X)-S\sim_{\bQ} \text{effective}\ \bQ\text{-divisor}\}.$$

\begin{remark}\label{upper mu0}Clearly, we have $0<\mu_0\leq \frac{m_0}{\iota(m_0)}$. 
If $|-m_0K_X|$ is composed with a pencil, for all $k$ such that $|-kK_X|\succeq |-m_0K_X|$ and $|-kK_X|$ is also composed with a pencil, we have 
$$k\pi^*(-K_X)=\iota(k)S+F_{k}$$ for some effective
$\bQ$-divisor $F_{k}$ by Lemma \ref{pencils}, and hence $\mu_0\leq  \frac{k}{\iota(k)}$.
\end{remark}

By the assumption on $|-m_1K_X|$, we know that $|G|=|{M_{-m_1}}|_{S}|$ is a base point free linear system on $S$ and $h^0(S, G)\geq 2$. Denote by
$C$ a generic irreducible element of $|G|$. Since $m_1\pi^*(-K_X)\geq M_{-m_1}$, we have
$$m_1\pi^*(-K_X)|_S\equiv C+H$$ 
where $H$ is an effective
$\bQ$-divisor on $S$. 

We define two numbers which will be the key
invariants accounting for the birationality of $\varphi_{-m}$.
They are
\begin{align*}
\zeta{}&:=(\pi^*(-K_X)\cdot C)_Y=(\pi^*(-K_X)|_S\cdot C)_S\ \text{and}\\
\varepsilon(m){}&:=(m+1-\mu_0-m_1)\zeta.
\end{align*}
Note that $\zeta$ and $\varepsilon(m)$ are invariants under taking higher model of the resolution $Y$ by projection formula. Hence we can modify $\pi$ if necessary. 

While studying the birationality of $\varphi_{-m}$, we always need to check that the linear system
$\Lambda_m:=|K_Y+\roundup{(m+1)\pi^*(-K_X)}|$ satisfies the
following assumption for some integer $m>0$.

\begin{assum}\label{b asum}Keep the notation as above.
\begin{itemize}
\item [(1)] The linear system $\Lambda_m$
distinguishes different generic irreducible elements of $|M_{-m_0}|$
(namely, $\Phi_{\Lambda_m}(S')\neq \Phi_{\Lambda_m}(S'')$ for two
different generic irreducible elements $S'$, $S''$ of $|M_{-m_0}|$).

\item [(2)] The linear system ${\Lambda_m}|_{S}$ distinguishes different generic irreducible elements of the linear system $|G|=|{M_{-m_1}}|_{S}|$ on
$S$.
\end{itemize}
\end{assum}

The following is the key theorem in this section.

\begin{thm}[{cf. \cite[Theorem 3.5]{C}}]\label{b mb} Let $X$ be a weak $\bQ$-Fano
3-fold. Keep the notation as above. Let $m>0$ be an integer. If
Assumption \ref{b asum} is satisfied and $\varepsilon(m)>2$, then
$\varphi_{-m}$ is birational onto its image.
\end{thm}
\begin{proof} By Lemma \ref{b reduction}, we only need to prove the birationality of ${\Phi_{\Lambda_m}}$. Since Assumption \ref{b asum}(1) is satisfied, the usual birationality principle (see, for instance,  \cite[2.7]{CC2}) reduces the birationality of ${\Phi_{\Lambda_m}}$ to that of
${\Phi_{\Lambda_m}}|_S$ for a generic irreducible element $S$ of $|M_{-m_0}|$.
Similarly, due to Assumption \ref{b asum}(2), we only need to prove
the birationality of ${\Phi_{\Lambda_m}}|_C$ for a generic
irreducible element $C$ of $|G|$. Now we show how to restrict the linear system $\Lambda_m$ to $C$.


Now assume $\varepsilon(m)>0$. We can find a sufficiently large integer $n$ so that there exists a number $\mu_0^{(n)}\in \bQ^+$ with $0\leq \mu_0^{(n)}-\mu_0\leq \frac{1}{n}$, $\roundup{\varepsilon(m,n)}=\roundup{\varepsilon(m)}$ 
where $\varepsilon(m,n):=(m+1-\mu_0^{(n)}-m_1)\zeta$ 
and 
$$\mu_0^{(n)}\pi^*(-K_X)\sim_{\bQ} S+E^{(n)}$$ for an effective $\bQ$-divisor $E^{(n)}$. In particular, $\varepsilon(m,n)>0$, and $\varepsilon(m,n)>2$ if $\varepsilon(m)>2$. Re-modify the resolution $\pi$ in Subsection \ref{b setting} so that $E^{(n)}$ has simple normal crossing support.

For the given integer $m>0$, we have
\begin{align}\label{b subsys1}
|K_Y+\roundup{(m+1)\pi^*(-K_X)-E^{(n)}}|\preceq
|K_Y+\roundup{(m+1)\pi^*(-K_X)}|.
\end{align}
Since
$\varepsilon(m,n)>0$, the $\bQ$-divisor
$$(m+1)\pi^*(-K_X)-E^{(n)}-S\equiv (m+1-\mu_0^{(n)})\pi^*(-K_X)$$
is nef and big and thus $$H^1(Y,
K_Y+\roundup{(m+1)\pi^*(-K_X)-E^{(n)}}-S)=0$$ by
Kawamata--Viehweg vanishing theorem. Hence we have
surjective map
\begin{align}\label{b surj1}
H^0(Y,K_Y+\roundup{(m+1)\pi^*(-K_X)-E^{(n)}})\longrightarrow
H^0(S, K_S+L_{m,n}) 
\end{align} 
where
\begin{align}\label{b subsys2}
L_{m,n}:=(\roundup{(m+1)\pi^*(-K_X)-E^{(n)}}-S)|_S\geq
\roundup{\mathcal{L}_{m,n}}
\end{align}
and ${\mathcal
L}_{m,n}:=((m+1)\pi^*(-K_X)-E^{(n)}-S)|_S.$ Moreover, we have
$$m_1\pi^*(-K_X)|_S\equiv C+H$$
for an effective $\bQ$-divisor $H$ on $S$ by the setting. Thus the $\bQ$-divisor
$${\mathcal
L}_{m,n}-H-C\equiv (m+1-\mu_0^{(n)}-m_1) \pi^*(-K_X)|_S$$
is nef and big by $\varepsilon(m,n)>0$. By Kawamata--Viehweg vanishing theorem again, $$H^1(S,
K_S+\roundup{{\mathcal
L}_{m,n}-H}-C)=0.$$ Therefore, we have surjective map
\begin{align}\label{b surj2}
H^0(S, K_S+\roundup{{\mathcal
L}_{m,n}-H})\longrightarrow
H^0(C, K_C+D_{m,n})
\end{align} 
where
\begin{align}\label{b subsys3}
D_{m,n}:=\roundup{{\mathcal
L}_{m,n}-H-C}|_C\geq
\roundup{\mathcal{D}_{m,n}}
\end{align} 
and
$\mathcal{D}_{m,n}:=({\mathcal
L}_{m,n}-H-C)|_C$ with
$\deg\roundup{\mathcal{D}_{m,n}}\geq \roundup{\varepsilon(m,n)}$.

Now by relations (\ref{b subsys1})--(\ref{b subsys3}), to prove
the birationality of ${\Phi_{\Lambda_m}}|_C$, it is
sufficient to prove that $|K_C+\roundup{\mathcal{D}_{m,n}}|$ gives a
birational map. Clearly this is the case whenever $\varepsilon(m)>2$,
which in fact implies $\deg(\roundup{\mathcal{D}_{m,n}})\geq \roundup{\varepsilon(m,n)}\geq 3$ and $K_C+\roundup{\mathcal{D}_{m,n}}$ is very ample. 

We complete the proof.
\end{proof}

\begin{cor}\label{m00}  Keep the same notation as above. For any integer $m>0$, set
$$\varepsilon(m,0):=(m+1-\frac{m_0}{\iota(m_0)}-m_1)\zeta.$$ If  $\varepsilon(m,0)>0$, then
$$\Lambda_m|_S\succeq |K_S+L_m|$$
where  $L_m:=(\roundup{(m+1)\pi^*(-K_X)-\frac{1}{\iota(m_0)}F_{m_0}}-S)|_S$.
\end{cor} 
\begin{proof} {}First of all, relation (\ref{b subsys1}) reads 
\begin{align}\label{b0 subsys1}
|K_Y+\roundup{(m+1)\pi^*(-K_X)-\frac{1}{\iota(m_0)}F_{m_0}}|\preceq
|K_Y+\roundup{(m+1)\pi^*(-K_X)}|.
\end{align}
In fact, as long as  $\varepsilon(m,0)>0$, the front part of the proof of Theorem \ref{b mb} is valid. In explicit, surjective map (\ref{b surj1}) reads the following surjective map
\begin{align}\label{b0 surj1}
H^0(Y,K_Y+\roundup{(m+1)\pi^*(-K_X)-\frac{1}{\iota(m_0)}F_{m_0}})\longrightarrow
H^0(S, K_S+L_m) 
\end{align} 
where
\begin{align}\label{b0 subsys2}
L_m:=(\roundup{(m+1)\pi^*(-K_X)-\frac{1}{\iota(m_0)}F_{m_0}}-S)|_S.
\end{align}
Hence we have proved the statement.
\end{proof}

\subsection{Applications}\label{b section birationality}\ 

In order to apply Theorem \ref{b mb}, we need to verify Assumption \ref{b asum} and $\varepsilon(m)>2$ in advance, for which one of the crucial steps is to estimate the lower bound of $\zeta$.

\begin{prop}[{cf. \cite[Theorem 3.2]{C}}]\label{b zeta}Let $m>0$ be an integer. Keep the same notation as in Subsection \ref{b keythm}.
\begin{itemize}
\item[(i)] If $g(C)>0$ and $\varepsilon(m)>1$, then $\zeta\geq \frac{2g(C)-2+\roundup{\varepsilon(m)}}{m}$;

\item[(ii)] Moreover, if $g(C)>0$, then $$\zeta\geq \frac{2g(C)-1}{\mu_0+m_1};$$

\item[(iii)] If $g(C)=1$, then $\zeta\geq \frac{1}{r_{\max}}$, where $r_{\max}=\max\{r_i\in B_X\}$ is the maximum of local indices of singularities;

\item[(iv)] If $g(C)=0$, then $\zeta\geq 2$;

\item[(v)] If $h^0(-\nu K_X)>0$ for some integer $\nu$, then $\zeta\geq \frac{1}{\nu r_{\max}}$.
\end{itemize}
\end{prop}
\begin{proof}(i). In the proof of Theorem \ref{b mb}, if $g(C)>0$ and $\varepsilon(m)>1$ then
$|K_C+\roundup{\mathcal{D}_{m,n}}|$ is base point free with
$$\deg(K_C+\roundup{\mathcal{D}_{m,n}})\geq 2g(C)-2+\roundup{\varepsilon(m,n)}=2g(C)-2+\roundup{\varepsilon(m)}.$$
Denote by $\mathcal{N}_m$ the movable part of
$|K_S+\roundup{\mathcal{L}_{m,n}-H}|$. Noting the relations
(\ref{b subsys1})--(\ref{b subsys3}) while applying \cite[Lemma 2.7]{Camb}, we
get
$$m\pi^*(-K_X)|_S\geq {M_{-m}}|_S\geq \mathcal{N}_m$$
and ${\mathcal{N}_m}|_C\geq K_C+\roundup{\mathcal{D}_{m,n}}$ since
the latter one is base point free.
So we have 
$$m\zeta=m\pi^*(-K_X)|_S\cdot C\geq \mathcal{N}_m\cdot C\geq \deg(K_C+\roundup{\mathcal{D}_{m,n}}).$$
Hence 
$$m\zeta\geq
2g(C)-2+\roundup{\varepsilon(m)}.$$

(ii). Take $m'=\min\{m\mid\varepsilon(m)>1\}$, then (i) implies
$\zeta\geq \frac{2g(C)}{m'}$. We may assume that $m'> \mu_0+m_1$ otherwise $\zeta\geq \frac{2g(C)}{\mu_0+m_1}$. Hence 
\begin{align*}
\varepsilon(m'-1){}&=(m'-1+1-\mu_0-m_1)\zeta\\
{}&\geq (m'-\mu_0-m_1)\frac{2g(C)}{m'}.
\end{align*}
By the minimality of $m'$, it follows that $\varepsilon(m'-1)\leq 1$. Hence $m'\leq \frac{2g(C)}{2g(C)-1}(\mu_0+m_1)$. Then $$\zeta\geq \frac{2g(C)}{m'}\geq \frac{2g(C)-1}{\mu_0+m_1}.$$

(iii). If $g(C)=1$, then
\begin{align*}
\zeta={}&(\pi^*(-K_X)\cdot C)_Y=((-K_{Y}+E_{\pi})\cdot C)_Y\\
={}& (-(K_{Y}+S)\cdot C+S\cdot C+E_{\pi}\cdot C)_Y\\
={}& (-K_S\cdot C)_S+(S\cdot C+E_{\pi}\cdot C)_Y\\
={}& (C^2)_S+(S\cdot C+E_{\pi}\cdot C)_Y.
\end{align*}
Since $C$ is free on surface $S$,  $(C^2)_S$, $(S\cdot C)_Y$, and $(E_{\pi}\cdot C)_Y$ are non-negative.
Since $(C^2)_S$ and  $(S\cdot C)_Y$ are integers, we may assume $(C^2)_S=(S\cdot C)_Y=0$ otherwise $\zeta\geq 1$. Hence $\zeta=E_{\pi}\cdot C$.

On the other hand, take $q:W\rightarrow X$ is the resolution of isolated singularities and  we may assume that $Y$ dominates $W$ by $p:Y\rightarrow W$.
Then we write
$$
K_W=q^*K_X+\Delta.
$$
Here 
$$
\Delta=\sum\frac{a_i}{r_i}E_i
$$
where $E_i$ is the exceptional divisor over an isolated singular point of index $r_i$ for some $r_i\in B_X$ and $a_i$ is a positive integer.
Then 
$$
E_{\pi}=K_Y-p^*K_W+p^*\Delta.
$$
Take $r_{\text{max}}=\max\{r_i\}$. Then all the coefficients of $E_{\pi}$ are at least $\frac{1}{r_{\text{max}}}$ since $K_Y-p^*K_W$ is integral  effective and $${\rm Supp}(E_{\pi})={\rm Supp}(K_Y-p^*K_W+p^{-1}_*\Delta).$$
By $E_{\pi}\cdot C=\zeta>0$, we know that there is at least one component $E$ of $E_{\pi}$ such that  $E\cdot C>0$.
Then $E_{\pi}\cdot C\geq \frac{1}{r_{\text{max}}} E\cdot C\geq \frac{1}{r_{\text{max}}}.$

(iv). If $g(C)=0$, then
\begin{align*}
\zeta={}&(\pi^*(-K_X)|_S\cdot C)_S=((-K_{Y}+E_{\pi})|_S\cdot C)_S\\
\geq{}& (-K_{Y}|_S\cdot C)_S\geq (-K_S\cdot C)_S\geq -\deg(K_C)=2.
\end{align*}

(v).  If $h^0(-\nu K_X)>0$ for some integer $\nu$, then $-\nu K_X\sim D$ for some effective Weil divosor $D$. Similarly as (iii),  $\pi^*D$ is an effective $\bQ$-divisor with all the coefficients  at least $\frac{1}{r_{\text{max}}}$. By $\pi^*D\cdot C=\nu\zeta>0$, we know that there is at least one component $D_1$ of $\pi^*D$ such that  $D_1\cdot C>0$.
Then $\zeta=\frac{1}{\nu}\pi^*D\cdot C\geq \frac{1}{\nu r_{\text{max}}} D_1\cdot C\geq \frac{1}{\nu r_{\text{max}}}.$
\end{proof}

To verify Assumption \ref{b asum}(1), we have the following proposition.

\begin{prop}[{cf. \cite[Proposition 3.6]{C}}]\label{b a1} Let $X$ be a weak $\bQ$-Fano
3-fold. Keep the same notation as Subsection \ref{b keythm}. Then
Assumption \ref{b asum}(1) is satisfied for all
$$m\geq \begin{cases}
m_0+6, & {\rm if } \ m_0\geq 2;\\
2, &{\rm if } \ m_0=1.
\end{cases}$$
\end{prop}
\begin{proof}We have
\begin{align*} 
{}& K_Y+\roundup{(m+1)\pi^*(-K_X)}\\
\geq {}& K_Y+\roundup{(m-m_0+1)\pi^*(-K_X)+M_{-m_0}}\\
={}& (K_Y+\roundup{(m-m_0+1)\pi^*(-K_X)})+M_{-m_0}\\
\geq{}& M_{-m_0}.
\end{align*} 
The last inequality is due to $$
h^{0}(K_Y+\roundup{(m-m_0+1)\pi^*(-K_X)})=h^{0}(-(m-m_0)K_X)>0
$$
by Lemma \ref{b reduction} and \cite[Appendix]{C}, since $m-m_0\geq 6$ whenever $m_0\geq 2$ (resp. $\geq 1$ whenever
$m_0=1$).

When $f:Y\rightarrow \Gamma$ is of type $(f_{\text{np}})$,
\cite[Lemma 2]{T} implies that $\Lambda_m$ can distinguish different
generic irreducible elements of $|M_{-m_0}|$. When $f$ is of type
$(f_{\text{p}})$, since the rational (i.e. $\Gamma\cong \bP^1$)
pencil $|M_{-m_0}|$ can already separate different fibers of $f$,
$\Lambda_m$ can naturally distinguish different generic irreducible
elements of $|M_{-m_0}|$.
\end{proof}

It is slightly more complicated to verify Assumption \ref{b asum}(2).

\begin{lem}[{cf. \cite[Lemma 3.7]{C}}]\label{b S2} Let $T$ be a nonsingular projective surface with a base point free linear system $|G|$. Let $Q$ be
an arbitrary $\bQ$-divisor on $T$. 
Denote by $C$ a generic irreducible element of $|G|$. Then the linear system
$|K_T+\roundup{Q}+G|$ can distinguish different generic irreducible
elements of $|G|$ under one of the following conditions:
\begin{itemize}
\item[(i)] $|G|$ is not
composed with an irrational pencil of curves and $K_T+\roundup{Q}$ is effective;

\item[(ii)]$|G|$ is composed with an
irrational pencil of curves, $g(C)> 0$, and $Q$ is nef and big;

\item[(iii)] $|G|$ is composed with an
irrational pencil of curves, $g(C)=0$, $Q$ is nef and big, and $Q\cdot G>1$.
\end{itemize}\end{lem}

\begin{proof} The statement corresponding to (i) follows from \cite[Lemma 2]{T} and
the fact that a rational pencil can automatically separate its
different generic irreducible elements.

For situations (ii) and (iii), we pick a generic irreducible element
$C$ of $|G|$. Then, since $h^0(S, G)\geq 2$, $G\equiv sC$ for some
integer $s\geq 2$ and $C^2=0$. Denote by $C_1$ and $C_2$ two
irreducible elements of $|G|$ such that $C_1+C_2\leq |G|$. Then Kawamata--Viehweg vanishing theorem
gives the surjective map
$$
H^0(T,K_T+\roundup{Q}+G)\longrightarrow H^0(C,K_{C_1}+D_1)\oplus
H^0(C_2,K_{C_2}+D_2)
$$ 
where $D_i:=(\roundup{Q}+G-C_i)|_{C_i}$ with $\deg(D_i)\geq Q\cdot C_i>0$ for $i=1,2$.

If $g(C)> 0$, Riemann--Roch formula gives
$h^0(C_i, K_{C_i}+D_i)>0$ for $i=1,2$. Thus $|K_T+\roundup{Q}+G|$ can
distinguish $C_1$ and $C_2$.

If $g(C)=0$ and $Q\cdot C>1$, then $h^0(C_i, K_{C_i}+D_i)> 0$ for $i=1,2$. So $|K_T+\roundup{Q}+G|$ can also
distinguish $C_1$ and $C_2$.
\end{proof}

\begin{prop}[{cf. \cite[Propositions 3.8, 3.9]{C}}]\label{b a2} Let $X$ be a weak $\bQ$-Fano 3-fold. Keep the same notation as in Subsection
\ref{b keythm}. 
Then Assumption \ref{b asum}(2) is satisfied for all 
$$m\geq \begin{cases}
m_0+m_1+6, &{\rm if } \  m_0\geq 2;\\
m_1+2, &{\rm if } \ m_0=1.
\end{cases}$$
\end{prop}
\begin{proof}  Assuming $m\geq m_0+m_1$, we have $\varepsilon(m,0)>0$,  and Corollary \ref{m00} implies that  
$${\Lambda_m}|_S\succeq |K_S+L_{m}|.$$ 
It suffices to prove that $|K_S+L_{m}|$ can distinguish different generic irreducible elements of $|G|$.

For a suitable integer $m>0$, we have the following relations:
\begin{align*} 
K_S+L_m
={}& (K_Y)|_S+\roundup{(m+1)\pi^*(-K_X)-\frac{1}{\iota(m_0)}F_{m_0}}|_S\\
\geq{}& (K_Y+\roundup{(m+1-m_0-m_1)\pi^*(-K_X)})|_S+{M_{-m_1}}|_S.
\end{align*} 
Thus, if $|G|$ is not composed with an irrational pencil of curves,
$|K_S+L_m|$ can distinguish different irreducible elements provided that $$K_Y+\roundup{(m+1-m_0-m_1)\pi^*(-K_X)}$$ is effective, which holds for $m-m_0-m_1\geq 6$ whenever $m_0\geq 2$ (resp. $\geq 1$ whenever
$m_0=1$) by \cite[Appendix]{C}.

Assume $|G|$ is composed with an irrational pencil of curves. we
have
\begin{align*}
K_S+L_m \geq {}& K_S+\roundup{((m+1)\pi^*(-K_X)-\frac{1}{\iota(m_0)}F_{m_0}-S)|_S}\\
\geq{}& K_S+\roundup{((m-m_1+1)\pi^*(-K_X)-\frac{1}{\iota(m_0)}F_{m_0}-S)|_S}+{M_{-m_1}}|_S.
\end{align*} 
We can take $Q=((m-m_1+1)\pi^*(-K_X)-\frac{1}{\iota(m_0)}F_{m_0}-S)|_S$ in Lemma \ref{b S2} since $\varepsilon(m,0)>0$.

If $g(C)> 0$, Lemma \ref{b S2}(ii) implies that Assumption \ref{b asum}(2) is satisfied for $m\geq m_0+m_1$. 

If $g(C)=0$, by Lemma
\ref{b S2}(iii), we need the condition
$\varepsilon(m,0)=(m+1-\frac{m_0}{\iota}-m_1)\zeta=Q\cdot C>1$. But this holds automatically for $m\geq m_0+m_1$ by Proposition \ref{b zeta}(iv). 

We complete the proof.
\end{proof}

Now we can treat the birationality of $\varphi_{-m}$ using Theorem \ref{b mb}.

\begin{thm}[{cf. \cite[Theorems 4.1, 4.2, 4.5]{C}}]\label{b main} 
Let $X$ be a weak $\bQ$-Fano 3-fold.  Let $\nu_0$ be an integer such that $h^0(-\nu_0K_X)>0$. Keep the same notation as in Subsection \ref{b keythm}. Then $\varphi_{-m}$ is birational onto its image if one of the following holds:
\begin{itemize}
\item[(i)] $m\geq \max\{m_0+m_1+a(m_0), \rounddown{3\mu_0}+3m_1\}$;

\item[(ii)] $m\geq \max\{m_0+m_1+a(m_0), \rounddown{\frac{5}{3}\mu_0+ \frac{5}{3}m_1}, \rounddown{\mu_0}+m_1+2r_{\max}\}$;

\item[(iii)] $m\geq \max\{m_0+m_1+a(m_0), \rounddown{\mu_0}+m_1+2\nu_0r_{\max}\},$
\end{itemize}
where $a(m_0)= \begin{cases}
6, & {\rm if } \ m_0\geq 2;\\
1, &{\rm if } \ m_0=1.
\end{cases}$ 
\end{thm}

\begin{proof}
By Propositions \ref{b a1} and \ref{b a2}, Assumption \ref{b asum} is satisfied if $m\geq m_0+m_1+a(m_0)$.

By Proposition \ref{b zeta}(v), $\zeta\geq \frac{1}{\nu_0r_{\text{max}}}$. If $m\geq \rounddown{\mu_0}+m_1+2\nu_0r_{\text{max}}$, then $\varepsilon(m)=
(m+1-\mu_0-m_1)\zeta>2$, which implies (iii). 

For (i) and (ii), we will discuss on the value of $g(C)$.
\smallskip

{\bf Case 1.} $g(C)=0.$

By Proposition \ref{b zeta}(iv), $\zeta\geq 2$. If $m\geq \rounddown{\mu_0}+m_1+1$, then $\varepsilon(m)=
(m+1-\mu_0-m_1)\zeta>2$.
\smallskip

{\bf Case 2.} $g(C)\geq 2.$

By Proposition \ref{b zeta}(ii), $\zeta\geq \frac{3}{\mu_0+m_1}$. If $m\geq \rounddown{\frac{5}{3}\mu_0+\frac{5}{3}m_1}$ then $\varepsilon(m)\geq
(m+1-\mu_0-m_1)\zeta>2$.
\smallskip

{\bf Case 3.} $g(C)=1$.

By Proposition \ref{b zeta}(ii), $\zeta\geq  \frac{1}{\mu_0+m_1}$. If $m\geq \rounddown{3\mu_0}+3m_1$, then $\varepsilon(m)=(m+1-\mu_0-m_1)\zeta>2$. So we have proved (i). 
On the other hand, by Proposition \ref{b zeta}(iii), $\zeta\geq \frac{1}{r_{\text{max}}}$. If $m\geq \rounddown{\mu_0}+m_1+2r_{\text{max}}$, then $\varepsilon(m)=(m+1-\mu_0-m_1)\zeta>2$. Thus (ii) is proved. 
\end{proof}

In practice, usually we just  use the fact $\mu_0\leq \frac{m_0}{\iota(m_0)}\leq m_0$. For very few cases, we will utilize a precise upper bound of $\mu_0$ rather than $m_0$ by Remark \ref{upper mu0}. 

Theorem \ref{b main} is optimal in some cases due to the following examples.
\begin{example}[{\cite[List 16.6]{Fletcher}}]
Consider general weighted hypersurface $X_{6d}\subset\mathbb{P}(1,a,b,2d,3d)$ where $1\leq a \leq b$ and $d=a+b$ such that $X_{6d}$ is a $\mathbb{Q}$-Fano $3$-fold with $r_{\text{max}}=d$. By {\cite[List 16.6]{Fletcher}}, there are exactly 12 such examples. Then $\varphi_{-3d}$ is birational onto its image but $\varphi_{-(3d-1)}$ is not by the structure. On the other hand, We can take $\nu_0=1$, $m_0=\mu_0=a$, and $m_1=b$, then
\begin{align*}
3d={}&\rounddown{3\mu_0}+3m_1\\
={}&\rounddown{\mu_0}+m_1+2r_{\text{max}}\\
={}&\rounddown{\mu_0}+m_1+2\nu_0r_{\text{max}}.
\end{align*}
Hence Theorem \ref{b main} tells that $\varphi_{-m}$ is birational onto its image for all $m\geq 3d$.
\end{example}

Theorem \ref{b main} directly implies the following result which generalizes a result of Fukuda \cite[Main theorem]{F}.

\begin{cor} Let $X$ be a weak $\bQ$-Fano $3$-fold with Gorenstein singularities. Then $\varphi_{-m}$ is birational onto its image for all $m\geq 4$.
\end{cor}
\begin{proof}
By Reid's formula, $P_{-1}=\frac{1}{2}(-K_X^3)+3>3$. Hence we can take $m_0=\nu_0=1$. 

If $|-K_X|$ is not composed with a pencil, then we can take $m_1=1$ and $\mu_0\leq m_0=1$. The result follows directly from Theorem \ref{b main}(iii).

If $|-K_X|$ is  composed with a pencil, then $\mu_0\leq \frac{m_0}{\iota(m_0)}<\frac{1}{2}$.
By Reid's formula again, $P_{-2}=\frac{5}{2}(-K_X^3)+5>r_X(-K_X^3)2+1$.  We can take $m_1=2$
 by Corollary \ref{b non-pencil}. The result follows directly from Theorem \ref{b main}(iii).
\end{proof}

\subsection{Proof of Theorems \ref{birationality1} and \ref{birationality2}}\ 

Now we prove the main results on the birationality of $\varphi_{-m}$.

\begin{proof}[Proof of Theorem \ref{birationality1}] To apply Theorem \ref{b main}, we always use the fact $\mu_0\leq m_0$. By \cite[Theorem 1.1]{CC} and  Theorem \ref{main1}, we can take $m_0\leq 8$ and $m_1\leq 10$ to apply Theorem \ref{b main}(i) and (ii). Hence $m_0+m_1+6\leq 24$ and $\frac{5}{3}(m_0+m_1)\leq 30$. By Theorem \ref{b main}, it is sufficient to prove that  either $3m_0+3m_1\leq 39$ or $m_0+m_1+2r_{\text{max}}\leq 39$ holds if we choose suitable $m_0$ and $m_1$. (Note that $\nu_0$ is not used in this proof.) 
\smallskip

{\bf Case 1.} $P_{-1}\geq 2.$

In this case, we can take $m_0=1$ and $m_1\leq 6$ by 
Theorem \ref{p2}. Hence  $3m_0+3m_1\leq 21$.
\smallskip

{\bf Case 2.} $P_{-1}=1.$

Recall the proof of Theorem \ref{p1}. We take $m_0=n_0$. If $m_0\leq 5$, then  we can take $m_1\leq 7$ and hence $3m_0+3m_1\leq 36$.
Similarly, if $m_0=6$ and if we can take $m_1\leq 7$, then $3m_0+3m_1\leq 39$.

If $m_0=n_0=6$ and $\delta_1(X)=8$, we can take $m_1=8$. Theorem \ref{k2} implies that
$$
P_{-1}=P_{-2}=P_{-3}=P_{-4}=P_{-5}=1, P_{-6}=P_{-7}=2. 
$$
Then  $n_{1,2}^0=2$, $n_{1,3}^0=2$, $n_{1,4}^0=2-\sigma_5$, $\epsilon_5=2-\sigma_5$, $0=\epsilon_6=3-\epsilon$.  Hence $\epsilon=3$ and $\sigma_5\leq 2$, and this implies $(\sigma_5,n_{1,5}^0)=(2,1)$. Then $\epsilon_5=0$ and $B^{(5)}(B)=\{ 2\times(1,2), 2\times(1,3), (1,5),(1,s)\}$ for some $s\geq 6$.  This implies $\epsilon_7=0$ since there are no further packings. On the other hand, $\epsilon_7=2-2\sigma_5+2n_{1,5}^0+n_{1,6}^0$. Hence $n_{1,6}^0=0$ and $B^{(7)}=\{ 2\times(1,2), 2\times(1,3), (1,5),(1,s)\}$ with $s\geq 7$. Since $B^{(7)}$ admits no prime packings, $B=B^{(7)}$. By inequalities (\ref{kwmt}) and (\ref{vol}), $s$ can only be $8, 9,10$. Hence $m_0+m_1+2{r_{\text{max}}}\leq  6+8+2\times 10=34$.

If $m_0=n_0\geq 7$, then
$$
P_{-1}=P_{-2}=P_{-3}=P_{-4}=P_{-5}=P_{-6}=1. 
$$
The proof of Theorem \ref{p1} implies $B^{(5)}=\{(1,2), (2,5), (1,3),(1,4),(1,s)\}$ with $s \ge 6$. Since $\gamma(B^{(5)}) >0$, we have $s\leq 11$. Noting that $B$ is dominated by $B^{(5)}$, we see $r_{\text{max}}\leq 11$. By Theorem \ref{p1}, we can take $m_0\leq 8$ and $m_1\leq 9$. Hence $m_0+m_1+2{r_{\text{max}}}\leq  8+9+2\times 11=39$.
\smallskip

{\bf Case 3.} $P_{-1}=P_{-2}=0.$

By the proof of Theorems \ref{p0} and   \ref{ex}, if $B$ is of type No.1, No.2 or No.4, then we have $r_{\text{max}}\leq 10$ and may take $m_0=8$, $m_1=10$. Hence $m_0+m_1+2{r_{\text{max}}}\leq  8+10+2\times 10=38$. If $B$ is of type No.5--No.6, then we have $r_{\text{max}}\leq 7$ and make take $m_0=7$, $m_1=8$. Hence $m_0+m_1+2{r_{\text{max}}}\leq  7+8+2\times 7=29$. If $B$ is of type No.7--No.23, then we can take $m_0=m_1=6$. Hence $3m_0+3m_1\leq 36$. Now the remaining case is type No.3:
$$
\{5\times (1,2), 2\times(1,3),(3,11)\}.
$$
Recall that $P_{-8}=P_{-9}=2$ and $-4K_X\sim E$ is a prime divisor by the proof of Theorem \ref{ex}(i).
By the proof of Theorem \ref{k2}, $|-8K_X|$ has no fixed part.  If $|-8K_X|$ and $|-9K_X|$ are composed with a same pencil, we can write
\begin{align*}
|-8K_X|{}&=|S'|,\\
|-9K_X|{}&=|S'|+F,
\end{align*}
where $F$ is the fixed part. This implies that
$$-K_X\sim -9K_X-(-8K_X)=F,$$ 
which contradicts $P_{-1}=0$. Hence $|-8K_X|$ and $|-9K_X|$ are composed with different pencils, and we can take $m_0=8$, $m_1=9$, and  $m_0+m_1+2r_{\max}=39$. 
\smallskip

{\bf Case 4.} $P_{-1}=0, P_{-2}>0$.

By \cite[Proposition 3.10, Case 1]{CC}, we can take $m_0=6$. We can take $m_1$ the same as in the proof of Theorems \ref{p0}  and \ref{ex}. If $m_1\leq 6$, then $3m_0+3m_1\leq 36$. If $m_1\geq 7$, observing Subsubcases II-3-ii and  II-3-iii in the proof of Theorem \ref{p0}, we can see that $r_{\text{max}}\leq 11$ holds for any such basket except $$B_d=\{4\times(1,2),(6,13),(1,5)\}.$$ Except for $B_d$, we have $m_0+m_1+2{r_{\text{max}}}\leq  6+8+2\times 11=36$. Now we deal with $B_d$. We claim that we can take $m_1=7$. Recall that $$P_{-1}=P_{-3}=0, P_{-2}=P_{-4}=P_{-5}=1, P_{-6}=P_{-7}=2.$$ Clearly $|-6K_X|$ and $|-7K_X|$ are both composed with pencils. We only need to show that they are composed with different pencils. To the contrary, we assume that $|-6K_X|$ and $|-7K_X|$ are composed with the same pencil.  If $-2K_X\sim D$ is a prime divisor, then by the proof of Theorem \ref{k2}, $|-6K_X|$ has no fixed part. By assumption, we can write
\begin{align*}
|-6K_X|{}&=|S'|,\\
|-7K_X|{}&=|S'|+F,
\end{align*}
where $F$ is the fixed part. This implies that
$$-K_X\sim -7K_X-(-6K_X)=F,$$ 
a contradiction. Hence $-2K_X\sim D$ is not a prime divisor. By the proof of Theorem \ref{ex}(ii), $D=E_1+E_2$ with $E_1$ and $E_2$ different prime divisors. Also we can write
\begin{align*}
|-6K_X|{}&=|S'|+a_6E_1,\\
|-7K_X|{}&=|S'|+F,
\end{align*}
where $a_6E_1$ and $F$ are the fixed parts with $a_6\leq 3$. If $a_6\leq 1$, then
$$
S'\sim 3(E_1+E_2)-a_6E_1\geq 2E_1+2E_2\sim -4K_X.
$$
This implies $|-7K_X|\succeq |-4K_X|$, which contradicts $P_{-3}=0$. If $a_6=3$, as in the proof of Theorem \ref{k1}, take $m=6$ and $E=E_1$ or $2E_1$ or $3E_1$, inequality (\ref{qq}) must fail for some singularity $P$ in $B_d$. Clearly, such an offending singularity $P$ must be ``$(6,13)$''. By equality (\ref{qq2}), the local index $i_P(E)$ of $E$ can only be $9$ or $11$ since  inequality  (\ref{qq}) holds for other $0\leq i\leq 12$ and $(b,r)=(6,13)$. But clearly the local index $i_P(E_1)$, $i_P(2E_1)$, and $i_P(3E_1)$ can not be in the set $\{9,11\}$ simultaneously,  a contradiction. 
Finally we consider the case $a_6=2$. Write $-5K_X\sim B$ a fixed divisor. Then
$$
B+S'+2E_1\sim-5K_X-6K_X\sim -4K_X-7K_X\sim 2E_1+2E_2+ S'+F,
$$
that is, $B\sim 2E_2+F$. Obviously, $F\not =0$. As in the proof of Theorem \ref{k1}, take $m=5$ and $E=E_1$ or $2E_1$, inequality (\ref{qq}) must fails for some singularity $P$ in $B_d$. Clearly, such an offending singularity $P$ must be ``$(6,13)$''. By equality (\ref{qq2}), the local index $i_P(E)$ of $E$ can only be $10$ or $11$ since  inequality  (\ref{qq}) holds for other $0\leq i\leq 12$ and $(b,r)=(6,13)$. But clearly the local index $i_P(E_1)$, $i_P(2E_1)$ can not be in the set $\{10,11\}$ simultaneously,  a contradiction.

We complete the proof.
\end{proof}


\begin{proof}[Proof of Theorem \ref{birationality2}] 
We shall apply Theorem \ref{b main} to treat arbitrary weak $\bQ$-Fano $3$-folds. We will choose suitable $m_0$ and $m_1$. Unless otherwise specified, we will use the fact $\mu_0\leq m_0$.
\smallskip

{\bf Case I.} $P_{-2}=0. $

In this case, the possible baskets are classified in Proposition \ref{list}. From the list we can take $m_0=8$. We have $r_X\leq 210$, $-K_X^3\geq \frac{1}{84}$, and $r_{\text{max}}\leq 14$. By Proposition \ref{b thm1} with $t=8$, we can take $m_1=38$. Hence by Theorem \ref{b main}(ii),   $\varphi_{-m}$ is birational onto its image for all $m\geq 76$.
\smallskip

{\bf Case II.} $r_{\text{max}}\geq 14. $

Write Reid's basket $B_X$ as 
$$\{(b_i,r_i)\mid i=1,\cdots, s; 0<b_i\leq \frac{r_i}{2};b_i \text{ is coprime to } r_i\}.$$ Recall that $r_X=\text{l.c.m.}\{r_i\mid i=1,\cdots, s\}$ and that 
\begin{align*}
\sum_i\big(r_i-\frac{1}{r_i}\big)\leq 24
\end{align*}
by inequality (\ref{kwmt2}).  We recall the sequence $\mathcal{R}=(r_i)_i$ from the proof of Proposition \ref{bound index}. Denote by $\tilde{r}_1=r_{\text{max}}$ the largest value in $\mathcal{R}$, by $\tilde{r}_2$ the second largest value, and by $\tilde{r}_3$, $\tilde{r}_4$ the third, the forth, and so on. For instance, if $\mathcal{R}=(2,3,4,4,5,5)$, then $\tilde{r}_1=5$, $\tilde{r}_2=4$, $\tilde{r}_3=3$, and $\tilde{r}_4=2$. If the value $\tilde{r}_j$ does not exist by definition, then we set $\tilde{r}_j=1$. In the previous example, we have $\tilde{r}_5=1$. 

Clearly $r_{\text{max}}\leq 24$. We will compute an explicit bound for $r_X$.

If $r_{\text{max}}\geq 23$, then by inequality (\ref{kwmt2}), there are no more values in  $\mathcal{R}$. Hence $r_X\leq 24$.



If $20\leq r_{\text{max}} \leq 22$, then by inequality (\ref{kwmt2}), $\rr_2\leq 4$. Hence $$r_X\leq {\rm l.c.m}(r_{\text{max}},  4, 3, 2)= 132.$$

If $ r_{\text{max}} =19$, then by inequality (\ref{kwmt2}), $\rr_2\leq 5$, and at most one of $3,4,5$ can be in $\mathcal{R}$. Hence $r_X\leq19\times 5\times 2= 190$.

If $ r_{\text{max}} =18$, then by inequality (\ref{kwmt2}), $\rr_2\leq 6$, and at most one of $3,4,5,6$ can be in $\mathcal{R}$. Hence $r_X\leq18\times 5 = 90$.

If $ r_{\text{max}} =17$, then by inequality (\ref{kwmt2}), $\rr_2\leq 7$. If $\rr_2\geq 5$, then by inequality (\ref{kwmt2}), $\rr_3\leq 2$ and hence $\rr_X\leq 17\times 7\times 2=238$.  If $\rr_2\leq 4$, then  $r_X\leq {\rm l.c.m}(17,4,3,2)=204$. 

If $ r_{\text{max}} =16$, then by inequality (\ref{kwmt2}), $\rr_2\leq 8$. If $\rr_2\geq 6$, then  by inequality (\ref{kwmt2}), $\rr_3\leq 2$ and hence $r_X\leq 16\times 7=112$.  If $\rr_2\leq 5$, then $r_X\leq {\rm l.c.m}(16,5,4,3,2)=240$. 

If $ r_{\text{max}} =15$, then by inequality (\ref{kwmt2}), $\rr_2\leq 9$. If $\rr_2\geq 6$, then  by inequality (\ref{kwmt2}), $\rr_3\leq 3$ and hence $r_X\leq {\rm l.c.m}(r_{\text{max}},\rr_2,3,2)\leq 15\times 7\times 2=210$.  If $\rr_2\leq 5$, then $r_X\leq {\rm l.c.m}(15,5,4,3,2)=60$. 

If $ r_{\text{max}} =14$, then by inequality (\ref{kwmt2}), $\rr_2\leq 10$. If $\rr_2\geq 8$, then  by inequality (\ref{kwmt2}), $\rr_3\leq 2$ and hence $\rr_X\leq   14\times 9 =126$.  If $\rr_2\leq 7$, then $r_X$ divides ${\rm l.c.m}(14,6, 5,4,3,2)=420$. But  by inequality (\ref{kwmt2}), $5,4,3$ can not be in $\mathcal{R}$ simultaneously, hence $r_X<420$. In particular, $r_X\leq 210$.

In summary, when $r_{\text{max}}\geq 14$, we have $r_X\leq 240$. 

We can take $m_0=8$ by \cite[Theorem 1.1]{CC}. We have $r_X\leq 240$, $-K_X^3\geq \frac{1}{240}$ (note that $r_XK_X^3$ is an integer), and $r_{\text{max}}\leq 24$. If $r_{\text{max}}\leq 22$, by Proposition \ref{b thm1} with $t=6$, we can take $m_1=44$. Hence by Theorem \ref{b main}(ii),   $\varphi_{-m}$ is birational onto its image for all $m\geq 96$. If $r_{\text{max}}=23$ or $24$, by Proposition \ref{b thm1} with $t=2$, $r_X\leq 24$, $-K_X^3\geq \frac{1}{24}$, we can take $m_1=37$. Hence by Theorem \ref{b main}(ii),   $\varphi_{-m}$ is birational onto its image for all $m\geq 93$.
\smallskip

{\bf Case III.} $r_{\text{max}}<14$ and $P_{-1}>0$.

In this case, $\nu_0=1$ and  by \cite[Theorem 1.1]{CC}, we can take $m_0=8$. 

If $r_X\leq 660$  and $r_{\text{max}}\leq 12$, then by Proposition \ref{b thm1} with $t=15$, $r_{\text{max}}\leq 12$, and $-K_X^3\geq \frac{1}{330}$, we can take $m_1=65$. Hence by Theorem \ref{b main}(iii),   $\varphi_{-m}$ is birational onto its image for all $m\geq 97$.

If $r_X\leq 660$  and $r_{\text{max}}= 13$, Then $\rr_2\leq 11$. If $\rr_2\geq 9$, then $\rr_3\leq 2$ and $r_X\leq 286$.
If $\rr_2=8$, then $\rr_3\leq 3$ and $r_X\leq 312$. If $\rr_2=7$, then $\rr_3\leq 4$ and $3,4$ can not be in $\mathcal{R}$ simultaneously, hence $r_X\leq 546$.  If $\rr_2\leq6$, then $r_X$ divides $780$ and hence $r_X\leq 390$ by Proposition \ref{bound index}. In summary, $r_X\leq 546$. By Proposition \ref{b thm1} with $t=10$, $r_{\text{max}}= 13$, and $-K_X^3\geq \frac{1}{330}$, we can take $m_1=61$. Hence by Theorem \ref{b main}(iii),   $\varphi_{-m}$ is birational onto its image for all $m\geq 95$.

If $r_X>660$, then $r_X=840$ and $r_{\text{max}}=8$. By Theorem \ref{main2}, we can take $m_1=71$. Hence by Theorem \ref{b main}(iii),   $\varphi_{-m}$ is birational onto its image for all $m\geq 95$.
\smallskip

{\bf Case IV.} $r_{\text{max}}<14$, $P_{-1}=0$, and $P_{-2}>0$.

In this case, $\nu_0=2$ and  by \cite[Proposition 3.10, Case 1]{CC}, we can take $m_0=6$. 

If $P_{-4}=1$, then $P_{-2}=1$. By the proof of Theorem \ref{p0} (note that the arguments on baskets are valid without assuming $\rho=1$),  we are exactly in the situation $(P_{-3},P_{-4})=(0,1)$, corresponding to the last paragraph of Subsubcase II-3-iii of Theorem \ref{p0}.  In fact, the possible baskets are classified in the following list:
\begin{align*}
{}&\{9\times (1,2), (1,3), (1,7) \},\\
{}&\{8\times (1,2), (2,5), (1,7) \},\\
{}&\{8\times (1,2), (2,5), (1,6) \},\\ 
{}&\{7\times (1,2), (3,7), (1,6) \},\\
{}&\{6\times (1,2), (4,9), (1,6)\},\\
{}&\{7\times (1,2), (3,7), (1,5) \},\\ 
{}&\{6\times (1,2), (4,9), (1,5) \},\\
{}&\{5\times (1,2), (5,11), (1,5) \},\\
{}&\{4\times (1,2), (6,13), (1,5) \}.
\end{align*}
Hence in this case $r_X\leq 130$, $-K_X^3\geq \frac{1}{130}$, and $r_{\text{max}}\leq 13$.  By Proposition \ref{b thm1} with $t=7$, we can take $m_1=37$. Hence by Theorem \ref{b main}(iii),   $\varphi_{-m}$ is birational onto its image for all $m\geq 95$.

Hence, from now on, we assume that $P_{-4}>1$.  So we may take $m_0=4$.

If $r_{\text{max}}\leq 8$, then $r_X$ divides ${\rm l.c.m}(8, 7,6,5,4,3,2)=840$. Suppose $r_X<840$, then $r_X\leq 420$.  By Proposition \ref{b thm1} with $t=20$ and $-K_X^3\geq \frac{1}{330}$, we can take $m_1=54$. Hence by Theorem \ref{b main}(iii),   $\varphi_{-m}$ is birational onto its image for all $m\geq 90$. Suppose $r_X=840$, then $\mathcal{R}=(3,5,7,8)$ or $(2,3,5,7,8)$ as we have seen in the proof of Proposition \ref{bound index}. However, 
\begin{align}
P_{-1}{}&=\frac{1}{2}(-K_X^3)-\sum\frac{b_i(r_i-b_i)}{2r_i}+3\notag\\
{}&>3-\frac{1}{4}-\frac{2}{6}-\frac{6}{10}-\frac{12}{14}-\frac{15}{16}>0\label{840},
\end{align}
a contradiction.

The above argument reminds us to find a condition corresponding to $P_{-1}=0$.
Assume that $2$ is not in $\mathcal{R}$, then 
\begin{align*}
P_{-1}{}&=\frac{1}{2}(-K_X^3)-\sum\frac{b_i(r_i-b_i)}{2r_i}+3\\
{}&>3-\frac{1}{8}\sum\big(r_i-\frac{1}{r_i}\big)\geq 0,
\end{align*}
a contradiction. Hence, $2\in \mathcal{R}$.

Consider the case $r_{\text{max}}=9$. If  $\rr_2\leq 6$, then $r_X\leq {\rm l.c.m}(9, 6,5,4,3,2)=180$. If $\rr_2=8$, then by inequality (\ref{kwmt2}) and $2\in \mathcal{R}$, $\rr_2\leq 5$ and $r_X\leq {\rm l.c.m}(9,8,5,4,3,2)=360$.  If  $\rr_2=7$ and $5\not\in \mathcal{R}$, then $$r_X\leq \lcm(9,7,6,4,3,2)=252.$$  If  $\rr_2=7$ and $5\in \mathcal{R}$, then $6 \not\in \mathcal{R}$ and $r_X$ divides ${\rm l.c.m}(9,7,5,4,3,2)=630$. In summary, $r_X\leq 360$ or $r_X=630$. Whenever $r_X\leq 360$, by Proposition \ref{b thm1} with $t=12$ and $-K_X^3\geq \frac{1}{330}$, we can take $m_1=50$. Hence by Theorem \ref{b main}(iii),   $\varphi_{-m}$ is birational onto its image for all $m\geq 90$.  Whenever $r_X=630$, then $2,5,7,9$ must be in $\mathcal{R}$. Hence $\mathcal{R}=(2,5,7,9)$ or $(2,2,5,7,9)$ by inequality (\ref{kwmt2}). In this case, by arguing as inequality (\ref{840}), $B_X$ can only be $\{2\times (1,2), (2,5),(3,7),(4,9)\}$. We will choose suitable $m_1$ and modify the upper bound of $\mu_0$. Since $P_{-4}=2$, $|-4K_X|$ is composed with a pencil. 
Note that $P_{-7}=10$ and $P_{-3}=1$.
If  $|-7K_X|$ is not composed with a pencil,  then we can take $m_1=7$. By Theorem \ref{b main}(ii),   $\varphi_{-m}$ is birational onto its image for all $m\geq 29$.
If  $|-7K_X|$ is also composed with a pencil, then we know $\mu_0\leq \frac{7}{9}$ by Remark \ref{upper mu0}. 
Also we can see $P_{-61}=5294>r_X(-K_X^3)61+1$ by direct computation using Reid's formula where $-K_X^3=\frac{43}{315}$. Hence we can take $m_1=61$ by Corollary \ref{b non-pencil}. 
Hence by  Theorem \ref{b main}(iii), $\varphi_m$ is
birational for all $m\geq 97$.

Consider the case $r_{\text{max}}=10$.  If  $\rr_2\leq 6$, then $$r_X\leq {\rm l.c.m}(10, 6,5,4,3,2)=60.$$ If  $\rr_2=7$, then $r_X$ divides  ${\rm l.c.m}(10, 7, 5, 4, 3, 2)=420$, but $3,4$ can not  be in $\mathcal{R}$ simultaneously, hence $r_X\leq 210$. If  $\rr_2=8$, then $r_3\leq 4$ and  $r_X\leq {\rm l.c.m}(10, 8, 4,3,2)=120$. If  $\rr_2=9$, then $\rr_3\leq 3$ and  $r_X\leq {\rm l.c.m}(10, 9 ,3,2)=90$. Hence in summary, $r_X\leq {210}$. By Proposition \ref{b thm1} with $t=10$ and $-K_X^3\geq \frac{1}{210}$, we can take $m_1=39$. Hence by Theorem \ref{b main}(ii),   $\varphi_{-m}$ is birational onto its image for all $m\geq 71$. 

Consider the case $r_{\text{max}}=11$.  If  $\rr_2=10$, then $\rr_3\leq 2$ and $r_X\leq110$. If  $\rr_2=9$ or $8$, then $\rr_3\leq 3$ and $r_X\leq 264$.
If  $\rr_2=7$, then $\rr_3\leq 4$ and $3,4$ can not be in $\mathcal{R}$ simultaneously, hence $r_X\leq 308$ or $r_X={\rm l.c.m}(11, 7,3,2)=462$. If  $\rr_2=6$, then  $5,4$ can not be in $\mathcal{R}$ simultaneously, hence $r_X\leq {\rm l.c.m}(11, 6,5 ,3,2)=330$.  If  $\rr_2\leq5$, then   $r_X$ divides ${\rm l.c.m}(11, 5,4 ,3,2)=660$. In summary, $r_X\leq 330$ or $r_X=462$ or $r_X=660$. 
Whenever $r_X=660$, then $2,3,4,5,11$ must be in $\mathcal{R}$. Hence $\mathcal{R}=(2,3,4,5,11)$ by inequality (\ref{kwmt2}). By arguing as inequality (\ref{840}), this implies $P_{-1}>0$, a contradiction. Whenever $r_X\leq 330$, by Proposition \ref{b thm1} with $t=13$ and $-K_X^3\geq \frac{1}{330}$, we can take $m_1=48$. Hence by Theorem \ref{b main}(ii),   $\varphi_{-m}$ is birational onto its image for all $m\geq 86$. If $r_X=462$, then $2,3,7,11$ must be in $\mathcal{R}$. Hence $\mathcal{R}=(2,3,7,11)$ or $(2,2,3,7,11)$ by inequality (\ref{kwmt2}). By arguing as inequality
(\ref{840}), $B_X$ can only be $\{2\times (1,2),(1,3),(3,7),(5,11)\}$. In this case we can prove $P_{-52}=2612>r_X(-K_X^3)52+1$ by direct computation using Reid's formula where $-K_X^3=\frac{50}{462}$. Hence we can take $m_1=52$. By Theorem \ref{b main}(ii),   $\varphi_{-m}$ is birational onto its image for all $m\geq 93$.

Consider the case $r_{\text{max}}=12$. Then $\rr_2\leq 10$ and at most one of $5,6,7,8,9,10$ will be in  $\mathcal{R}$. Hence $r_X\leq 84$. By Proposition \ref{b thm1} with $t=5$ and $-K_X^3\geq \frac{1}{84}$, we can take $m_1=37$. Hence by Theorem \ref{b main}(ii),   $\varphi_{-m}$ is birational onto its image for all $m\geq 68$.

Finally, consider the case $r_{\text{max}}=13$. Then $\rr_2\leq 9$. If $\rr_2=9$ or $8$, then $\rr_3\leq 2$ and $r_X\leq 234$.  If $\rr_2=7$, then $\rr_3\leq 3$ and $r_X=546$ or $182$. If $\rr_2\leq 6$, then $r_X$ divides $780$ and hence $r_X\leq 390$ by Proposition \ref{bound index}. In summary, $r_X\leq 390$ or $r_X=546$. Whenever $r_X\leq 390$, by Proposition \ref{b thm1} with $t=12$ and $-K_X^3\geq \frac{1}{330}$, we can take $m_1=52$. Hence by Theorem \ref{b main}(ii),   $\varphi_{-m}$ is birational onto its image for all $m\geq 93$. Whenever $r_X=546$, then $\mathcal{R}=(2,3,7,13)$. Argue as inequality (\ref{840}), $B_X$ can only be $\{(1,2),(1,3),(3,7),(6,13)\}$. 
We will choose suitable $m_1$ and modify the upper bound of $\mu_0$. Since $P_{-4}=2$, $|-4K_X|$ is composed with a pencil. 
Note that $P_{-10}=21$ and $P_{-6}=5$.
If  $|-10K_X|$ is not composed with a pencil,  then we can take $m_1=10$. By Theorem \ref{b main}(ii),   $\varphi_{-m}$ is birational onto its image for all $m\geq 40$.
If  $|-10K_X|$ is also composed with a pencil, then we know $\mu_0\leq \frac{1}{2}$ by Remark \ref{upper mu0}. Also we can prove $P_{-57}=3540>r_X(-K_X^3)57+1$ by direct computation using Reid's formula where $-K_X^3=\frac{61}{546}$. Hence we can take $m_1=57$. By Theorem \ref{b main}(ii),   $\varphi_{-m}$ is birational onto its image for all $m\geq 95$.

We complete the proof.
\end{proof}

\end{document}